\documentclass[12pt, reqno]{amsart}
\usepackage{amsmath,amsfonts,amssymb,amsthm}
\usepackage{anysize}
\marginsize{2cm}{2cm}{2cm}{2cm}
\usepackage{color}
\usepackage{enumitem}

\usepackage[pageref]{backref}
\usepackage[colorlinks=true, linkcolor=blue, citecolor=red, urlcolor=blue, backref=page]{hyperref}
\renewcommand*{\backrefalt}[4]{%
  \ifcase #1 %
    (Not cited.)%
  \or
    (Cited on page~\hyperlink{page.#2}{#2}.)%
  \else
    (Cited on pages~\hyperlink{page.#2}{#2}.)%
  \fi
}

\usepackage{setspace}
\usepackage{tikz}
\usetikzlibrary{graphs}
\usetikzlibrary{arrows}
\usetikzlibrary{positioning}
\usetikzlibrary{calc,fit,shapes.geometric}

\def\u{\mathfrak{u}}

\def\g{\mathfrak{g}}

\def\hcx{\{J_{\alpha}\}}

\def\C{\mathbb{C}}
\def\R{\mathbb{R}}
\def\Q{\mathbb{Q}}
\def\Z{\mathbb{Z}}
\def\N{\mathbb{N}}
\def\H{\mathbb H}

\def\al{\alpha}

\def\e{\operatorname{e}}
\def\d{\operatorname{d\!}}
\def\Ker{\operatorname{Ker}}
\def\Im{\operatorname{Im}}

\def\ad{\operatorname{ad}}
\def\tr{\operatorname{tr}}

\def\alt{\raise1pt\hbox{$\bigwedge$}}

\theoremstyle{plain}
\newtheorem{theorem}{\bf Theorem}[section]
\newtheorem{corollary}[theorem]{\bf Corollary}
\newtheorem{proposition}[theorem]{\bf Proposition}
\newtheorem{lemma}[theorem]{\bf Lemma}

\theoremstyle{definition}
\newtheorem{definition}[theorem]{\bf Definition}
\newtheorem{example}[theorem]{\bf Example}

\theoremstyle{remark}
\newtheorem{remark}[theorem]{\bf Remark}

\newcommand{\ri}{{\rm (i)}}
\newcommand{\rii}{{\rm (ii)}}
\newcommand{\riii}{{\rm (iii)}}

\newcommand{\vale}[1]{\textcolor{violet}{#1}}

\title[Betti and Hodge numbers]{Betti and Hodge numbers of solvmanifolds arising from integer polynomials}
\author{Adrián Andrada}
\email{adrian.andrada@unc.edu.ar}
\author{Valentina Chaves}
\email{vchaves@mi.unc.edu.ar}
\date{}

\address{FAMAF, Universidad Nacional de C\'ordoba and CIEM-CONICET, Av. Medina Allende s/n, Ciudad Universitaria, X5000HUA C\'ordoba, Argentina}

\begin{document}

\begin{abstract}
We study the de Rham cohomology of three families of completely solvable almost abelian solvmanifolds—called basic, complex, and hypercomplex—constructed from a monic integer polynomial with positive distinct roots whose product equals 1, following the work of Andrada and Barberis. Under two algebraic restrictions on such polynomials—the full rank and quasi full rank conditions—we compute the Betti numbers and Poincaré polynomials of these manifolds. Moreover, we study the Dolbeault cohomology of the complex solvmanifolds by identifying them with generalized Nakamura manifolds recently introduced by Cattaneo and Tomassini. Assuming a suitable condition on the lattice, we compute their Hodge numbers, which exhibit a combinatorial structure related to Pascal's triangle in the full rank setting, and are described by explicit generating polynomials in the quasi full rank case.
\end{abstract}

\maketitle

\section{Introduction}

The study of solvmanifolds has played a fundamental role in differential geometry, serving as an invaluable source of explicit examples of compact manifolds equipped with various geometric structures. Within this broad class, almost abelian solvmanifolds stand out for their tractability. Recall that a solvmanifold $\Gamma\backslash G$ is called almost abelian if the simply connected Lie group $G$ is almost abelian, that is, its Lie algebra $\g$ has a codimension one abelian ideal; equivalently, $\g$ can be written as $\g=\R e_0\ltimes_A \R^d$, where the matrix $A\in \mathfrak{gl}(d,\R)$ encodes the adjoint action of $e_0$ on the abelian ideal $\R^{d}$. The Lie algebra $\g$ is also called almost abelian. It follows that the topology and geometry of these solvmanifolds can be largely understood through the algebraic properties of the defining matrix $A$. 

Many authors have recently made important contributions to the subject (see, for instance, \cite{Av, C-M, CM, FP1, FP2, Fr1, HO, Mo}, among others). In particular, complex structures on almost abelian Lie algebras were studied in \cite{LRV} and more recently in \cite{ABDGH} (where more detailed information on the Jordan form of the matrix $A$ was provided) and in \cite{AABRW} (where the nilpotent case was analyzed in depth, including formulas for the Betti and Hodge numbers of the associated complex nilmanifolds). On the other hand, the existence and main properties of hypercomplex structures on such Lie algebras were discussed in \cite{AB1, AB2}. In the latter work \cite{AB2}, the authors provided a method to construct completely solvable solvmanifolds equipped with invariant complex or hypercomplex structures, starting from a monic integer polynomial $p$ of degree $n \geq 2$ whose roots satisfy the following conditions: they are positive, distinct, and their product is equal to 1; such a polynomial is said to belong to $\Delta_n$. The complex solvmanifold associated with $p$ has dimension $2n+2$, whereas the hypercomplex solvmanifold has dimension $4n+4$.

In this paper, we study in detail the de Rham cohomology of the complex and hypercomplex solvmanifolds associated with a polynomial $p$ in $\Delta_n$. In fact, we also analyze a third solvmanifold associated with $p$, of dimension $n+1$, which we call \textit{basic}. A crucial fact for our purposes is that, when dealing with completely solvable Lie groups, Hattori's theorem provides an isomorphism between the de Rham cohomology of the solvmanifold and the Chevalley-Eilenberg cohomology of its Lie algebra. In this way, determining the Betti numbers reduces to a combinatorial problem on left-invariant forms. Specifically, in our case, this amounts to counting the vanishing partial sums among the logarithms of the roots of the polynomial $p$.

To obtain closed formulas, we impose certain algebraic restrictions on the polynomials: the \textit{full rank condition} (introduced by T.\ Payne in \cite{Pay}) and the \textit{quasi full rank condition} for self-reciprocal polynomials of even degree. Roughly speaking, a polynomial $p \in \Delta_n$ satisfies the full rank condition if no proper subset of its roots has product equal to 1, whereas it satisfies the quasi full rank condition if it is self-reciprocal and the only proper subsets of its roots whose product is 1 are of the form $\{r, r^{-1}\}$, or unions of them. Under the full rank condition, we are able to explicitly determine the Betti numbers of the basic, complex, and hypercomplex solvmanifolds; in particular, we show that the cohomology of the basic solvmanifold is minimal among all solvmanifolds of the same dimension. In the quasi full rank setting, we again determine the Betti numbers in the basic case, while those of the complex and hypercomplex solvmanifolds are encoded as the coefficients of explicitly computed Poincaré polynomials.

Moreover, although completely solvable solvmanifolds admit neither Kähler nor hyper-Kähler metrics (according to \cite[Main Theorem]{Hasegawa}), we show that our complex and hypercomplex solvmanifolds nevertheless exhibit remarkable cohomological similarities with these geometries: the complex solvmanifolds have even odd-indexed Betti numbers, as in compact Kähler manifolds, while the Betti numbers of our hypercomplex solvmanifolds satisfy Salamon's identity \cite{Sal1,Sal}, a topological constraint originally established for compact hyper-Kähler manifolds.

Finally, we turn to the study of the Dolbeault cohomology in the complex case. We show that the complex solvmanifolds considered in this work can be naturally identified with Nakamura manifolds recently introduced in \cite{CT}. This correspondence allows us to apply general results from \textit{loc. cit.} on their Dolbeault cohomology and derive explicit expressions for the Hodge numbers under a suitable condition on the lattice defining the solvmanifold. In the full rank setting, we obtain closed formulas for the Hodge numbers and describe the resulting Hodge diamond, which exhibits a combinatorial structure closely related to Pascal's triangle. In the quasi full rank case, we encode the Hodge numbers through explicit generating polynomials.

A distinctive aspect of the families studied here is that their cohomological structure is determined by arithmetic properties of the associated integer polynomials. This arithmetic-geometric correspondence yields a unified setting in which explicit formulas can be obtained in arbitrary dimensions. Beyond the computations themselves, it reveals how subtle relations among the roots of integer polynomials may impose strong constraints on the topology and complex geometry of the resulting solvmanifolds.

The paper is organized as follows. In Section 2, we collect preliminaries on complex and hypercomplex structures on manifolds, properties of almost abelian solvmanifolds, and the Chevalley-Eilenberg cohomology of diagonal almost abelian Lie algebras. In Section 3, building on the constructions developed in \cite{AB2}, we introduce the three families of completely solvable almost abelian Lie algebras (basic, complex, and hypercomplex) constructed from a polynomial $p \in \Delta_n$ and review the existence of lattices for their associated simply connected Lie groups. We also recall the full rank condition, expressing it in terms of logarithmic parameters for our purposes, and present the quasi full rank condition for self-reciprocal polynomials. Section 4 is devoted to computing the de Rham cohomology of the basic $(n+1)$-dimensional solvmanifolds, explicitly determining their Betti numbers (see Propositions \ref{prop: betti_base_case} and \ref{prop: betti_reciprocal}). In Sections 5 and 6, we extend these computations to $(2n+2)$-dimensional complex solvmanifolds (Theorems \ref{thm: cohomology_complex} and \ref{thm: qfr_generating_function}) and $(4n+4)$-dimensional hypercomplex solvmanifolds (Theorems \ref{thm: cohomology_hyp} and \ref{thm: qfr_generating_function_hyp}), obtaining their Poincaré polynomials. We also compare their topological properties with those of compact Kähler and hyper-Kähler manifolds. Section 7 explores the Dolbeault cohomology of the complex solvmanifolds, computing their Hodge numbers and Hodge diamonds (Theorems \ref{thm: hodge_complex} and \ref{thm: hodge_generating_function}). Section 8 presents explicit examples of polynomials that do not satisfy either the full rank or quasi full rank conditions, illustrating the resulting increase in the Betti and Hodge numbers. Finally, the Appendix provides number-theoretic proofs ensuring the existence, for every integer ($n \geq 2$), of full rank polynomials in $\Delta_n$, as well as quasi full rank polynomials.

\subsection*{Acknowledgments}
A.\ Andrada was partially supported by CONICET, SECYT-UNC and FONCYT, and V.\ Chaves was partially supported by an undergraduate research fellowship from Consejo Interuniversitario Nacional (Argentina). The authors thank Jonas Deré for his valuable assistance with integer polynomials and for pointing out reference \cite{Pay}, which proved crucial for this work, and Marco Freibert for his useful comments.

\section{Preliminaries}    

\subsection{Complex and hypercomplex manifolds}
A complex structure on a differentiable manifold $M$ is an automorphism $J$ of the tangent bundle $TM$ satisfying $J^2=-I $  and the 
integrability condition $N_{J}(X,Y) =0$ for all  vector fields $X,Y$ on $M,$ where $N_J$ is the Nijenhuis tensor:
\begin{equation*}  N_{J}(X,Y) =
[X,Y]+J([JX,Y]+[X,JY])-[JX,JY].  \label{nijen} 
\end{equation*}
Recall that the integrability of $J$ is equivalent to the existence of an atlas  on $M$ such that the transition functions are holomorphic maps \cite{NN}. 

A hypercomplex structure on $M$ is a triple of complex structures $\hcx$, $\al=1,2,3$, on $M$ satisfying the following conditions: 
\begin{equation}  \label{quat}
J_1J_2=-J_2J_1=J_3.
\end{equation}
Then $M$ has a family of complex structures 
$J_{y}=y_1J_1+y_2J_2+y_3J_3$ parameterized by points $y = (y_1,y_2,y_3)$ in the unit sphere $S^2 \subset \R^3$. It follows from \eqref{quat} that $T_pM$, for each $p\in M$, has an $\H$-module structure, where $\H$ denotes the quaternions; in particular, $\dim M \equiv 0 \; \pmod{4}$.

Given a hypercomplex structure $\hcx$ on $M$, there is a unique torsion-free connection $\nabla$ on $M$ such that $\nabla J_\al=0, \; \al =1,2,3$. It is called the Obata connection \cite{Ob}. Its holonomy group $\operatorname{Hol}(\nabla)$ is therefore contained in the quaternionic  general linear group $\operatorname{GL}(n, \H)$. 

A hyperhermitian structure on $M$ is a pair $(\hcx ,g )$ where $\hcx$ is a hypercomplex structure and $(J_{\al},g )$ is Hermitian for $\al =1,2,3$.  
An interesting subclass of hyperhermitian structures is given by hyper-K\"ahler structures \cite{Cal}, which are hyperhermitian structures such that $(J_\al , g)$ is K\"ahler for $\al =1,2,3$, that is, the  K\"ahler forms $\omega_\al$ associated to $(J_\al,g)$ are closed, $\al =1,2,3$. In this case, the Levi-Civita connection coincides with the Obata connection, and its holonomy group is contained in $\operatorname{Sp}\,(n)$, where $\dim M=4n$.  Since $\operatorname{Sp}\,(n)\subset \operatorname{SU}\,(2n)$, hyper-K\"ahler metrics are Ricci-flat.

\subsection{Almost abelian solvmanifolds} A \textit{solvmanifold} is a compact quotient $\Gamma\backslash G$, where $G$ is a simply connected solvable Lie group and $\Gamma$ is a discrete subgroup of $G$. Such a subgroup $\Gamma$ is called a \textit{lattice} of $G$. When $G$ is nilpotent and $\Gamma\subset G$ is a lattice, the compact quotient $\Gamma\backslash G$ is known as a nilmanifold.

It follows that $\pi_1(\Gamma\backslash G)\cong \Gamma$ and  $\pi_n(\Gamma\backslash G)=0$ for $n>1$. Furthermore, solvmanifolds are determined up to diffeomorphism by their fundamental groups. In fact:

\begin{theorem}\cite{Mos}\label{thm:solv-isom}
If $\Gamma_1$ and $\Gamma_2$ are lattices in simply connected solvable Lie groups 
$G_1$ and $G_2$, respectively, and $\Gamma_1$ is isomorphic to $\Gamma_2$, then $\Gamma_1 \backslash G_1$ is diffeomorphic to $\Gamma_2 \backslash G_2$.
\end{theorem}

A solvable Lie group $G$ is called \textit{completely solvable} if the adjoint operators $\ad_x:\g\to\g$, with $x\in \g=\operatorname{Lie}(G)$, have only real eigenvalues. 
The conclusion of the previous theorem can be strengthened when both solvable Lie groups $G_1$ and $G_2$ are completely solvable. Indeed, this is the content of Saito's rigidity theorem:

\begin{theorem}\cite{Sai}\label{thm:Saito}
Let $G_1$ and $G_2$ be simply connected completely solvable Lie groups and $\Gamma_1 \subset G_1, \, \Gamma_2\subset G_2$ lattices. Then every isomorphism
$f: \Gamma_1 \to \Gamma_2$ 
extends uniquely to an isomorphism of Lie groups $F: G_1 \to G_2$.
\end{theorem}

Moreover, solvmanifolds of completely solvable Lie groups have a very useful property concerning their de Rham cohomology. Indeed, Hattori \cite{Hat} proved that the natural inclusion 
\[
\alt^* \g^* \hookrightarrow \Omega^*(\Gamma\backslash G),
\] 
with $G$ completely solvable, induces an isomorphism 
\begin{equation}\label{deRham}
H^*(\g) \cong H^*_{dR}(\Gamma\backslash G).
\end{equation}
That is, the de Rham cohomology of the solvmanifold can be computed in terms of left-invariant forms. In particular, $H^*_{dR}(\Gamma\backslash G)$ does not depend on the lattice $\Gamma$.
The isomorphism \eqref{deRham} was previously proved for nilmanifolds by Nomizu \cite{Nom}.

\begin{remark} \label{rem: first betti number}
    In the general case, the natural inclusion $\alt^* \g^* \hookrightarrow \Omega^*(\Gamma\backslash G)$ induces just an injective map $H^*(\g) \hookrightarrow H^*_{dR}(\Gamma\backslash G)$. Since $\g \neq [\g,\g]$ for solvable $\g$ and $H^1(\g) \cong \g/[\g,\g]$, we get that the first Betti number of a solvmanifold satisfies $b_1(\Gamma\backslash G)\geq b_1(\g) \geq 1$.
\end{remark}

\begin{remark}
Recall that the Poincaré polynomial of an $n$-dimensional compact manifold is defined as $P_M(t) = \sum_{k=0}^{n} b_k(M) t^k$, where $b_k(M)$ denote the Betti numbers of $M$. Moreover, the Euler characteristic of $M$, denoted by $\chi(M)$, is defined as the alternating sum of its Betti numbers $\chi(M) = \sum_{k=0}^{n} (-1)^k b_k(M)$, that is, $\chi(M)=P_M(-1)$.

When $M$ is a solvmanifold $M=\Gamma\backslash G$, it is well known that $\chi(\Gamma\backslash G)=0$, since it is parallelizable. It follows that $P_{\Gamma\backslash G}(-1)=0$, and therefore its Poincaré polynomial is divisible in $\Z[t]$ by $(1+t)$.
\end{remark}

\begin{remark}
Let $G$ be a Lie group with Lie algebra $\g$. A complex structure $J$ on $G$ is said to be left-invariant if left translations by elements of $G$ are holomorphic maps. In this case, $J$ is determined by its value at the identity of $G$, which corresponds to a complex structure on $\g$. We point out that if $\Gamma $ is a lattice in $G$, any left-invariant complex structure on $G$ induces a complex structure on $\Gamma \backslash G$ which is called invariant. In this case, the natural projection $G \to \Gamma \backslash G$ is a local bihomolorphism. Left-invariant hypercomplex structures on $G$ (and invariant hypercomplex structures on $\Gamma \backslash G$) are defined similarly.
\end{remark}

Concerning the Dolbeault cohomology of solvmanifolds, we point out that if $\Gamma\backslash G$ is a nilmanifold, it is known that there is an isomorphism $H^{*,*}_{\overline{\partial}}(\g,J) \cong H^*_{\overline{\partial}}(\Gamma\backslash G,J)$ in several important cases (one of them when $J$ is nilpotent), and it is conjectured that this isomorphism always holds (see, for example, \cite{FRR, RTW}). However, this is far from being true for general solvmanifolds, even under the assumption of complete solvability (see, for instance, \cite{Kas}).

\medskip

We recall next that a  Lie group $G$ is said to be {\em almost abelian} if its Lie algebra $\g$ has a codimension one abelian ideal. Such Lie algebra will be called almost abelian and can be written as $\g= \R e_0 \ltimes \mathfrak{u}$, where $\mathfrak u$ is an abelian ideal of $\g$, and $\R$ is generated by some $e_0\notin \u$. After choosing a basis of $\u$, we may identify $\u$ with an abelian Lie algebra $\R^{d}$ and write $\g=\R e_0\ltimes_A \R^{d}$ for some $A\in \mathfrak{gl}(d,\R)$.

Accordingly, the Lie group $G$ can be written as a semidirect product $G=\R\ltimes_\varphi \R^{d}$, where the action is given by $\varphi(t)=e^{t A}$. Notice that a non-abelian almost abelian Lie group is $2$-step solvable, and it is nilpotent if and only if the operator $A$ is nilpotent.

Regarding the isomorphism classes of almost abelian Lie algebras, we have the following result, proved in \cite{Fr}.

\begin{lemma}\label{lem:ad-conjugate}
Two almost abelian Lie algebras $\g_1=\R e_1 \ltimes_{A_1} \R^d$ and $\g_2=\R e_2 \ltimes_{A_2}\R^d$ are isomorphic if and only if there exists $c\neq 0$ such that $A_2$ and $cA_1$ are conjugate. 
\end{lemma}

\begin{remark}\label{rem:nilp-sim}
It follows that two nilpotent almost abelian  Lie algebras as above are isomorphic if and only if $A_1$ and $A_2$ are conjugate, since for   any   nilpotent matrix $N$,   $cN$ and $N$ are conjugate  whenever $c\neq 0$.     
\end{remark}


In general, it is not easy to determine whether a given Lie group $G$ admits a lattice. A well known restriction is that if this is the case then $G$ must be unimodular (see, for instance, \cite{Mi}), i.e. the Haar measure on $G$ is left and right invariant, which is equivalent, when $G$ is connected, to  $\tr(\ad_x)=0$ for any $x$ in the Lie algebra $\g$ of $G$. In the nilpotent case, there is a well-known criterion  due to Malcev:

\begin{theorem}\cite{Mal}\label{thm:Mal}
A simply connected nilpotent Lie group has a lattice if and only if its Lie algebra admits a basis with respect to which the structure constants are rational.
\end{theorem}

On the other hand, there is a criterion for the existence of lattices on almost abelian Lie groups  which will prove very useful in forthcoming sections:

\begin{proposition}\label{prop: Bock}\cite{Bo}
Let $G=\R\ltimes_\varphi\R^d$ be a unimodular almost abelian Lie group. Then $G$ admits a lattice if and only if there exists  $t_0\neq 0$ such that $\varphi(t_0)$ is conjugate to a matrix in $\operatorname{SL}(d,\Z)$. In this situation, a lattice is given by $\Gamma=t_0 \Z\ltimes P\mathbb Z^{d}$, where $P\in \operatorname{GL}(d,\R)$ satisfies $P^{-1}\varphi(t_0)P\in \operatorname{SL}(d,\Z)$. 
\end{proposition}

Note that if $E:=P^{-1}\varphi(t_0)P$ then $\Gamma\cong \Z\ltimes_E \Z^d$, where the group multiplication in this last group is given by 
\[ (m,(p_1,\ldots,p_d))\cdot (n,(q_1,\ldots, q_d))=(m+n,(p_1,\ldots,p_d)+E^m(q_1,\ldots, q_d)). \]

\subsection{Cohomology of diagonal almost abelian Lie algebras}
An almost abelian Lie algebra $\g=\R e_0\ltimes_A \R^d$ is called \textit{diagonal} if $A$ is diagonalizable over $\R$. According to Lemma \ref{lem:ad-conjugate}, we may simply assume that $A$ is a diagonal matrix, 
\begin{equation}\label{eq: A-diagonal}
    A=\operatorname{diag}(\lambda_1,\ldots, \lambda_d),
\end{equation} 
for some $\lambda_i\in \R$, in some basis $\{e_1,\ldots,e_d\}$ of $\R^d$. The Lie bracket in $\mathfrak{g}$ is completely determined by
\begin{equation*}
    [e_0, e_i] = \lambda_i e_i \quad \text{for } 1 \leq i \leq n
\end{equation*}
If $\{e^0, e^1, \dots, e^d\}$ denotes the dual basis of $\mathfrak{g}^*$, then the exterior derivative $\d\,\colon\g^*\to \alt^2 \g^*$ is given by
\begin{align}\label{eq:d-1forma}
  \d e^0 & =0, \\
  \d e^i & =-\lambda_i e^{0}\wedge e^i, \quad 1\leq i \leq d.\nonumber
\end{align} 

To compute $H^*(\mathfrak{g})$, we must understand the action of the exterior derivative in $k$-forms for $k\geq 2$. For convenience, we will use the notation $e^{j_1 \dots j_k} := e^{j_1} \wedge \dots \wedge e^{j_k}$ for the basis elements of $\alt^k \mathfrak{g}^*$, where $j_1 < j_2 < \dots < j_k$.

\begin{lemma} \label{lemma: formula_d}
Let $\g = \mathbb{R} e_0 \ltimes_{A} \mathbb{R}^d$ be a diagonal almost abelian Lie algebra with $A$ as in \eqref{eq: A-diagonal}. Then the exterior derivative $\d\,\colon \alt^k \mathfrak{g}^* \to \alt^{k+1} \mathfrak{g}^*$ acts on a basis element $e^{j_1 \dots j_k}$ as follows:
    \begin{enumerate}
        \item[$\ri$] If $j_1 = 0$, then $\d e^{0 j_2 \dots j_k} = 0$.
        \item[$\rii$] If $j_1 \neq 0$, then 
        \begin{equation}\label{eq: formula for d}
            \d e^{j_1 \dots j_k} = -\left( \sum_{i=1}^k \lambda_{j_i} \right) e^{0 j_1 \dots j_k}.
        \end{equation}
    \end{enumerate}
\end{lemma}

\begin{proof}
    We proceed by induction on the degree $k$ of the form. For $k=1$, this is equation \eqref{eq:d-1forma}.
        
    Assume now that the formula holds for all $(k-1)$-forms ($k > 1$) on $\g$. Let $e^{j_1 \dots j_k}$ be a basis element of $\alt^k \g^*$, that can be written as $e^{j_1 \dots j_k} = e^{j_1} \wedge \alpha$, where $\alpha = e^{j_2 \dots j_k}$ is a $(k-1)$-form. Using the Leibniz rule, we get:
    \[ \d e^{j_1 \dots j_k} = \d\, (e^{j_1} \wedge \alpha) = \d e^{j_1} \wedge \alpha - e^{j_1} \wedge \d\alpha.\]
    Let us analyze two cases, depending on the index $j_1$:

    \begin{itemize}
        \item Case $j_1 = 0$: 
        Since $\d e^0 = 0$, the first term vanishes. By inductive hypothesis applied to $\alpha$ (which does not contain $0$), $\d\alpha = - \left( \sum_{i=2}^k \lambda_{j_i} \right) e^{0 j_2 \dots j_k}$. Then:
    \[ \d e^{0 j_2 \dots j_k} = -e^0 \wedge \d\alpha = 0.\]
     This proves $\ri$ for $k$-forms.

        \item Case $j_1 > 0$: Here, none of the indices in the $k$-form are $0$. Recall from \eqref{eq:d-1forma} that $\d e^{j_1} = -\lambda_{j_1} e^{0 j_1}$. By the inductive hypothesis, $\d\alpha = -\left( \sum_{i=2}^k \lambda_{j_i} \right) e^{0 j_2 \dots j_k}$. Hence, we obtain:
    \begin{align*}
        \d e^{j_1 \dots j_k} &= (-\lambda_{j_1} e^{0 j_1}) \wedge e^{j_2 \dots j_k} - e^{j_1} \wedge \left( -\left( \sum_{i=2}^k \lambda_{j_i} \right) e^{0 j_2 \dots j_k} \right) \\
        &= -\lambda_{j_1} e^{0 j_1 \dots j_k} - \left( \sum_{i=2}^k \lambda_{j_i} \right) e^{0 j_1 \dots j_k}= -\left( \sum_{i=1}^k \lambda_{j_i} \right) e^{0 j_1 \dots j_k}.
    \end{align*}
     This proves $\rii$ for $k$-forms. 
     \end{itemize}
     Therefore, the lemma holds for all $k \geq 1$.
\end{proof}

The explicit formula obtained in Lemma \ref{lemma: formula_d} reveals that the action of the exterior derivative $\d$ is entirely determined by the partial sums of the eigenvalues $\lambda_i$. Specifically, for any basis element $e^{j_1 \dots j_k}$ not containing $e^0$, its differential vanishes if and only if the sum of the associated eigenvalues is zero. By systematically counting these zero-sum combinations, we obtain a general combinatorial formula for the Betti numbers of any completely solvable, unimodular, diagonal almost abelian Lie algebra. This reduction simplifies the computation of the de Rham cohomology for the solvmanifolds studied in the subsequent sections. 

We will use the following notation: For any $k \in \Z$, let $\ell_k$ denote the number of index subsets $I \subset \{1, \dots, d\}$ of size $|I| = k$ for which $\sum_{i \in I} \lambda_i = 0$. By convention, we set $\ell_0 = 1$ (corresponding to the empty sum) and $\ell_k = 0$ for $k < 0$ or $k > d$. Note that if $\g$ is unimodular, the eigenvalues of $A$ satisfy $\sum_{i=1}^d \lambda_i = 0$. Then, in this case, $\ell_d = 1$.

\begin{lemma} \label{lemma: betti_combinatorial}
    Let $\g = \R e_0 \ltimes_{A} \R^d$ be a unimodular diagonal almost abelian Lie algebra with $A$ as in \eqref{eq: A-diagonal}. Then, the Betti numbers $b_k = \dim H^k(\g)$ are given by:
    \begin{equation} \label{eq: betti_ell}
        b_k = \ell_k + \ell_{k-1}
    \end{equation}
    for all $0 \leq k \leq d+1$.
\end{lemma}

\begin{proof}
In this proof, we denote by ${\d\,}^k$ the linear operator $\d\,\colon \alt^k \g^*\to \alt^{k+1} \g^*$, for any $k$.

We compute the dimensions of the kernel and the image of the differential separately, knowing that $b_k = \dim(\Ker {\d\,}^k) - \dim(\Im {\d\,}^{k-1})$.
    
Let us first find $\dim(\Ker {\d\,}^k)$. A generic $k$-form $\omega \in \alt^k \g_p^*$ can be written as $\omega = \sum c_J e^J$, where $J = \{j_1, \dots, j_k\} \subset \{0, 1, \dots, d\}$ is a multi-index of size $k$ and $e^J = e^{j_1} \wedge \dots \wedge e^{j_k}$, with $j_1<\cdots<j_k$. If $\omega \in \Ker {\d\,}^k$, then:
    \[
        0=\d\omega = \sum_{|J|=k} c_J \d e^J = \sum_{0 \notin J} c_J \left( - \sum_{j \in J} \lambda_j \right) e^0 \wedge e^J
    \]
where the third equality follows from \eqref{eq: formula for d}. Since the $(k+1)$-forms $e^0 \wedge e^J$ (for distinct $J \subset \{1, \dots, d\}$) are linearly independent, the coefficient of each must be zero. Thus, for any $J$ of size $k$ not containing $0$, we have $c_J \left( \sum_{j \in J} \lambda_j \right) = 0$. 
Therefore, a basis for $\Ker {\d\,}^k$ is formed by:
    \begin{itemize}
        \item The $k$-forms $e^J$ such that $0 \in J$: the number of these forms is $\binom{d}{k-1}$.
        \item The $k$-forms $e^J$ such that $0 \notin J$ but for which $\sum_{j \in J} \lambda_j = 0$: by definition, the number of these forms is exactly $\ell_k$.
    \end{itemize}
    Summing these two contributions, we obtain $\dim(\Ker {\d\,}^k) = \binom{d}{k-1} + \ell_k$.

    Next, we compute $\dim(\operatorname{Im} {\d\,}^{k-1})$. The image is generated by the differentials $\d e^I$ for multi-indices $I$ of size $|I| = k-1$. According to Lemma \ref{lemma: formula_d}, if $0 \in I$, $\d e^I  = 0$. If $0 \notin I$, $\d e^I = -\left(\sum_{i \in I} \lambda_i\right) e^0 \wedge e^I$. 
    Thus, $\operatorname{Im}({\d\,}^{k-1})$ is generated by $k$-forms of the type $e^0 \wedge e^I$, which are linearly independent. A generator $\d e^I$ is zero if and only if $\sum_{i \in I} \lambda_i = 0$. 
    The total number of possible generators (subsets $I \subset \{1, \dots, d\}$ of size $k-1$) is $\binom{d}{k-1}$. Of these, exactly $\ell_{k-1}$ vanish. Therefore, $\dim(\operatorname{Im} {\d\,}^{k-1}) = \binom{d}{k-1} - \ell_{k-1}$.

    Combining our results, we conclude:
    \[
        b_k = \left( \binom{d}{k-1} + \ell_k \right) - \left( \binom{d}{k-1} - \ell_{k-1} \right) = \ell_k + \ell_{k-1}. \qedhere
    \]
\end{proof}

\section{Construction of the Solvmanifolds} \label{sec: construction}

In this section, we introduce the main objects of our study: three families of diagonal almost abelian solvmanifolds associated to a given polynomial $p \in \Delta_n$. These constructions are based on the work of Andrada and Barberis \cite{AB2}, where such Lie algebras were initially introduced to build solvmanifolds endowed with invariant complex or hypercomplex structures.

Following their approach, we define a \emph{basic} Lie algebra without any additional geometric structures, alongside its complex and hypercomplex counterparts. The detailed proofs guaranteeing that the simply connected Lie groups admit lattices, and that their quotients inherit the desired geometry, are fully established in \cite{AB2}. Nevertheless, we will explicitly provide the justification for the basic case here, as it relies on the exact same arguments.

For $n\in\N$, $n\geq 2$, let $\Delta_n$ denote the  subset of $\Z[x]$ given by all polynomials $p\in\Z[x]$ satisfying the following conditions:
\begin{enumerate}
\item[(i)] the degree of $p$ is $n$,
\item[(ii)] $p$ is monic,
\item[(iii)] $p$ has $n$ different real roots, all positive, and
\item[(iv)] $p(0)=(-1)^n$.
\end{enumerate}
We will also consider the following distinguished subset of  $\Delta_n$:
\begin{equation*}\label{eq:delta-prima}
\Delta_n ':=\{p\in \Delta_n : p(1)\neq 0\}.
\end{equation*}
We recall that $\Delta_n'$ is infinite for $n\geq 2$ (see \cite[Lemma 7.13]{AB2}). Other properties of the polynomials $p\in \Delta_n$ can be found in \cite[Section 7.2]{AB2}.

\begin{remark}\label{rem: Anosov}
    We point out that the polynomials in $\Delta'_n$ are examples of \textit{Anosov} polynomials, that is, monic polynomials with integer coefficients, with constant term equal to $\pm 1$, and without roots of modulus one (see \cite{Pay}).
\end{remark}


For any $p \in \Delta_n$, let $r_1, \dots, r_n$ be its positive real roots. Even though it is not strictly necessary, due to Lemma \ref{lem:ad-conjugate}, we assume that $r_1<r_2<\cdots < r_n$. We define the $n \times n$ diagonal matrix:
\begin{equation}\label{eq:log}
    A_p = \operatorname{diag}(\log r_1, \log r_2, \dots, \log r_n).
\end{equation}
Since $p(0) = (-1)^n$, we have $\operatorname{tr} A_p = \sum_{j=1}^n \log r_j = \log \left( \prod_{j=1}^n r_j \right) = \log(1) = 0$, which implies $A_p\in \mathfrak{sl}(n,\R)$. This matrix $A_p$ serves as the fundamental building block for the three families of completely solvable almost abelian Lie algebras that we will study:

\begin{itemize}
    \item \textbf{The basic case:} We define the basic $(n+1)$-dimensional Lie algebra as 
    \begin{equation*}
        \g_p = \R e_0 \ltimes_{A_p} \R^n.
    \end{equation*}
    
    \item \textbf{The complex case:} We define the $(2n+1) \times (2n+1)$ diagonal matrix 
    \begin{equation*}
        A_p^c = 0_1 \oplus A_p^{\oplus 2}.
    \end{equation*}
    Since $\operatorname{tr} A_p = 0$, it is clear that $\operatorname{tr} A_p^c = 0$, and consequently $A_p^c \in \mathfrak{sl}(2n+1, \R)$. The associated completely solvable almost abelian $(2n+2)$-dimensional Lie algebra is defined as 
    \begin{equation*}
        \g_p^c := \R e_0 \ltimes_{A_p^c} \R^{2n+1}.
    \end{equation*}
    This Lie algebra admits a complex structure, according to \cite{LRV}.

    \item \textbf{The hypercomplex case:} Similarly, we consider the $(4n+3) \times (4n+3)$ matrix 
    \begin{equation*}
        A_p^{h} = 0_3 \oplus A_p^{\oplus 4} \in \mathfrak{sl}(4n+3, \R)
    \end{equation*}
    and the corresponding unimodular Lie algebra is given by 
    \begin{equation*}
        \g_p^{h} := \R e_0 \ltimes_{A_p^{h}} \R^{4n+3}.
    \end{equation*}
    This Lie algebra admits a hypercomplex structure, according to \cite{AB1}.
\end{itemize}

\vspace{0.3cm}

The simply connected Lie groups associated to these algebras are semidirect products of the form $\R \ltimes_\varphi \R^d$, where the action $\varphi(t)$ is given by the exponential of the respective matrix. 

Following the reasoning proposed in \cite{AB2}, let us detail the basic case where $G_p = \R \ltimes_\varphi \R^n$ with $\varphi: \R \to \operatorname{SL}(n, \R)$ given by $\varphi(t) = \exp(t A_p)$. Setting $t=1$, we have that the characteristic and the minimal polynomial of the matrix $\varphi(1) = \exp(A_p)=\operatorname{diag}(r_1,\ldots,r_n)$ are both equal to $p$, therefore $\exp(A_p)$ is conjugate to the companion matrix $C_p$ of $p$. It follows from $p\in\Delta_n$ that $C_p\in \operatorname{SL}(n,\Z)$. Thus, according to Proposition \ref{prop: Bock}, $G_p$ admits a lattice $\Gamma_p:=\Z\ltimes_{\varphi(1)} Q_p \Z^{n}$, where $Q_p\in \operatorname{GL}(n,\R)$ satisfies $Q_p^{-1}\varphi(1) Q_p=C_p$. 

As proved in \cite{AB2}, a similar argument applies to the other cases: evaluating their respective group actions at $t=1$ gives the matrices $\exp(A_p^c)$ and $\exp(A_p^{h})$, which are conjugate to $I_1 \oplus C_p^{\oplus 2} \in \operatorname{SL}(2n+1, \Z)$ and $I_3 \oplus C_p^{\oplus 4} \in \operatorname{SL}(4n+3, \Z)$, respectively. Therefore, again by Proposition \ref{prop: Bock}, the simply connected Lie groups $G_p^c$ and $G_p^{h}$ also admit lattices $\Gamma_p^c$ and $\Gamma_p^{h}$. Furthermore, the corresponding quotients inherit invariant complex or hypercomplex structures.

To sum up the constructions of this section, we establish the following proposition covering all three cases:

\begin{proposition} \label{prop: three_solvmanifolds}
    For any $p \in \Delta_n$, the completely solvable unimodular Lie groups $G_p$, $G_p^c$, and $G_p^{h}$ admit lattices.
\end{proposition}

If $\Gamma, \Gamma^c$ and $\Gamma^h$ denote lattices in $G_p$, $G_p^c$, and $G_p^{h}$, respectively,  we obtain three families of compact almost abelian solvmanifolds associated to $p\in \Delta_n$:
    \begin{itemize}
        \item The basic $(n+1)$-dimensional solvmanifold $M_{p,\Gamma} = \Gamma \backslash G_p$,
        \item The complex $(2n+2)$-dimensional solvmanifold $M_{p,\Gamma}^c = \Gamma^c \backslash G_p^c$,
        \item The hypercomplex $(4n+4)$-dimensional solvmanifold $M_{p,\Gamma}^{h} = \Gamma^{h} \backslash G_p^{h}$.
    \end{itemize}
To keep the notation concise, we will usually denote these manifolds by $M_p,\, M_p^c,\, M_p^h$, omitting their explicit dependence on $\Gamma,\, \Gamma^c,\, \Gamma^h$.

\medskip

In the following sections, we study the de Rham cohomology of the solvmanifolds constructed above. To determine its Betti numbers, we apply Lemma \ref{lemma: betti_combinatorial} to the logarithmic parameters $s_i := \log r_i$, $i = 1, \dots, n$, where $r_1, \dots, r_n$ denote the roots of $p \in \Delta_n$ (see \eqref{eq:log}). Since this requires computing the integers $\ell_k$ counting vanishing partial sums of the parameters $s_i$, it is necessary to analyze the additive relations among these quantities. In particular, explicit computations become possible when such relations are sufficiently constrained. We will need the following definition.


\begin{definition} \label{def: full_rank}
    Let $p \in \Z[x]$ be a monic polynomial of degree $n \ge 2$ with constant term $(-1)^n$, and let $r_1, \dots, r_n \in \C$ be its roots. We say that $p$ satisfies the \textit{full rank condition} (or that $p$ \textit{has full rank}) if the only integral solutions to the equation
    \[
        r_1^{k_1} r_2^{k_2} \cdots r_n^{k_n} = 1
    \]
    are of the form $k_1 = k_2 = \dots = k_n$. 
\end{definition} 

The previous definition was introduced by Tracy Payne in \cite[Definition 3.5]{Pay}. Note that if $p$ satisfies the full rank condition then it is irreducible over $\Q$. In the case when $p \in \Delta_n$, the full rank condition can be written equivalently in the following way, which is more suitable for our purposes: 

Let $p \in \Delta_n$ have roots $r_1,\ldots, r_n$ and set $s_i=\log r_i$, so that 
\begin{equation}\label{eq:suma}
    \sum_{i=1}^n s_i = 0.
\end{equation}  
Then $p$ satisfies the \textit{full rank condition} (or $p$ \textit{has full rank}) if  the only $\Z$-linear relation among $\{s_1, \dots, s_n\}$ is the one given in \eqref{eq:suma}. That is, for any integer coefficients $k_i \in \Z$:
    \begin{equation}\label{eq:k_i iguales}
        \sum_{i=1}^n k_i s_i = 0 \iff k_1 = k_2 = \dots = k_n.
    \end{equation}
Note that this condition implies that there are no proper subsets of parameters whose sum is zero, nor are there any internal ``crossed'' relations between them.

\begin{remark}
Observe that if $p \in \Delta_n$ satisfies the full rank condition then $p(1)\neq 0$. Indeed, a root $r_i = 1$ would produce a nontrivial linear relation among the logarithmic parameters $s_i$, contradicting the full rank hypothesis. It follows that any full rank polynomial in $\Delta_n$ belongs to $\Delta'_n$.
\end{remark}

Next, we recall a result that provides a practical condition for verifying whether a polynomial satisfies the full rank condition:


\begin{proposition}\label{prop: tracy}
Let $r_1, \dots, r_n \in \C$ be the roots of an irreducible monic polynomial $p \in \Z[x]$ with constant term $(-1)^n$, and assume that $|r_j|\neq 1$ for all $j$ (in particular, $p$ is Anosov). Then $p$ satisfies the full rank condition in any of the following cases:
\begin{enumerate}
    \item[$\ri$] when precisely one of its roots has modulus greater than one, or
    \item[$\rii$] when the Galois group $G$ of $p$ acts doubly transitive on the set $\{r_1,\ldots,r_n\}$, or
    \item[$\riii$] when $n$ is prime.
\end{enumerate}
\end{proposition}

Items $\ri$ and $\rii$ of Proposition \ref{prop: tracy} were proved in \cite[Proposition 3.6]{Pay}, whereas item $\riii$ was proved in \cite[Proposition 5.1]{DW}.

\begin{remark}
    The distinguished root in Proposition \ref{prop: tracy}$\ri$ is a \textit{Pisot number} (also known as \textit{Pisot–Vijayaraghavan} or \textit{PV} number), that is, a real algebraic integer greater than 1, all of whose Galois conjugates are less than 1 in absolute value.
\end{remark}

A natural question is how the cohomology behaves once the strict requirement of nonzero partial sums is relaxed. A particularly interesting case arises from self-reciprocal polynomials. We recall their definition below.

\begin{definition}
    Given a polynomial $p$ of degree $n$, we define the associated polynomial $p^*$ given by:
    \[ p^*(x)=(-1)^nx^np(x^{-1}).  \]
We say that $p$ is \textit{self-reciprocal} if $p^*=(-1)^np$. That is, the coefficients of $p^*$ are the coefficients of $p$ in reverse order.
\end{definition}
We note that:
\begin{itemize}
    \item $p^*$ is monic precisely when  $p(0)= (-1)^n$,
    \item $(p^*)^*=p$,
    \item If $p(0) \neq 0$ and $r_1, \dots, r_n$ are the roots of $p$, then $r_1^{-1}, \dots, r_n^{-1}$ are the roots of $p^*$,
    \item $p\in \Z[x]$ if and only if $p^*\in\Z[x]$.
    \item $p\in \Delta_n$ if and only if $p^*\in\Delta_n$.
    \item If $p$ is self-reciprocal with odd degree, then $p(1)=0$. In particular, if $p\in \Delta_n$ is self-reciprocal and $n$ is odd, then $p\notin \Delta'_n$.
\end{itemize}

Since we are interested only in polynomials lying in $\Delta_n'$, in the next sections we will only consider self-reciprocal polynomials $p$ of even degree $n=2m$ such that $1$ is not one of its roots. In this case, its roots naturally appear in pairs $\{r, r^{-1}\}$, whose product is clearly 1; next we will define a family of self-reciprocal polynomials such that the only proper subsets of its roots whose product equals 1 are unions of these pairs. 

\begin{definition} \label{def: quasi_full_rank}
Let $p \in \Z[x]$ be a self-reciprocal Anosov polynomial of even degree $n=2m$, and let $r_1, \dots, r_m, r_1^{-1},\ldots, r_m^{-1} \in \C$ be its roots, with $|r_i|>1$ for all $i$. We say that $p$ satisfies the \textit{quasi full rank condition} (or that $p$ \textit{has quasi full rank}) if the only integral solutions to the equation
    \[
        r_1^{c_1} r_2^{c_2} \cdots r_m^{c_m} (r_1^{-1})^{c'_1}\ldots (r_m^{-1})^{c'_m} = 1
    \]
    are of the form $c_i = c'_i$ for all $i \in \{1, \dots, m\}$. Equivalently, for any integer coefficients $a_i \in \Z$,
\begin{equation} \label{eq: qfr_multiplicative_r}
    r_1^{a_1}\cdots r_m^{a_m}=1 \iff a_1 = \dots = a_m = 0.
\end{equation}
    
\end{definition}

For a self-reciprocal polynomial $p \in \Delta'_n$ of even degree $n = 2m$, this condition naturally translates to its additive parameters $s_i=\log r_i$. We say that $p$ has \textit{quasi full rank} if for any integer coefficients $c_i, c'_i \in \Z$, the additive relation
\begin{equation}\label{eq: QFR}
    \sum_{i=1}^m c_i s_i + \sum_{i=1}^m c'_i (-s_i) = 0
\end{equation}
holds if and only if $c_i = c'_i$ for all $i \in \{1, \dots, m\}$. This ensures that there are no ``cross-index'' cancellations between parameters belonging to different pairs.

\medskip

We will verify in the Appendix that for each $n\geq 2$ (respectively, $n=2m$ with $m \geq 1$) there exist polynomials in $\Delta'_n$ satisfying the full rank condition (respectively, polynomials in $\Delta'_n$ satisfying the quasi full rank condition).

\begin{remark}\label{rem:full rank}
    Regarding the full rank condition for the polynomial $p^*$, it is clear from Definition \ref{def: full_rank} that $p^*$ satisfies the full rank condition if and only if $p$ does.
\end{remark}

\section{De Rham cohomology of the basic solvmanifolds}

As stated in the preliminaries, the complete solvability of the Lie group $G_p$ implies, by Hattori's isomorphism \eqref{deRham}, that the de Rham cohomology of the solvmanifold $M_p = \Gamma \backslash G_p$ is isomorphic to the Chevalley-Eilenberg cohomology of its Lie algebra $\g_p$. Thus, the computation of the Betti numbers of $M_p$ reduces to the algebraic problem of determining $H^*(\g_p)$. We will do this assuming that $p \in \Delta'_n$ belongs to one of the two families considered above: (i) Full rank polynomials; (ii) Quasi full rank polynomials.

\subsection{The full rank case}

\begin{proposition} \label{prop: betti_base_case}
    Let $M_{p} = \Gamma \backslash G_p$ be a basic $(n+1)$-dimensional almost abelian solvmanifold constructed from a polynomial $p \in \Delta'_n$ satisfying the full rank condition. Then, the Betti numbers $b_k = \dim H^k_{dR}(M_p)$ are given by:
    \[
        b_k = \begin{cases} 
        1 & \text{if } k = 0, 1, n, n+1, \\
        0 & \text{if } 1 < k < n.
        \end{cases}
    \]
\end{proposition}

\begin{proof}
    By hypothesis, no proper partial sums vanish. The only vanishing sums are the empty sum ($\ell_0 = 1$) and the total sum, which is always zero since $A_p \in \mathfrak{sl}(n, \R)$ ($\ell_n = 1$). For all intermediate lengths $1 \leq k < n$, we have $\ell_k = 0$. By convention, $\ell_{-1} = 0$ and $\ell_{n+1} = 0$.
    Applying the formula $b_k = \ell_k + \ell_{k-1}$ from Lemma \ref{lemma: betti_combinatorial}, the result follows immediately.
\end{proof}

Proposition \ref{prop: betti_base_case} highlights a remarkable topological property of the basic solvmanifolds arising from a polynomial $p \in \Delta'_n$ that satisfies the full-rank condition: among all solvmanifolds of a fixed dimension, these basic solvmanifolds exhibit the minimal possible de Rham cohomology. Indeed, since any $(n+1)$-dimensional solvmanifold is compact, connected, and orientable, we automatically have $b_0 = b_{n+1} = 1$. Furthermore, Remark \ref{rem: first betti number} implies that $b_1 \geq 1$, which yields $b_{n} \geq 1$ by Poincaré duality. Consequently, the total Betti number satisfies $\sum_{k=0}^{n+1} b_k \geq 4$. According to Proposition \ref{prop: betti_base_case}, this minimum value is attained precisely by these basic solvmanifolds.

\begin{example}\label{ex: ejemplos basicos}
Let us consider the cases $n=2$ and $n=3$.

For $n=2$, it follows from \cite[Example 7.9]{AB2} that 
\[ \Delta_2'=\Delta_2=\{ h_m\in\Z[x] \colon h_m(x)=x^2-mx+1, \, m\in \Z\}. \]
With notation as above, we have that all polynomials $h_m$ satisfy $s_1\neq 0, s_2\neq 0$ and $s_1+s_2=0$, and this implies that they satisfy the full rank condition. Thus, the Betti numbers of $M_{h_m}$ are $b_k=1$ for $0\leq k\leq 3$.

For $n=3$, set $f_{r,s}(x)= x^3-rx^2+sx-1$, for $r,s\in \Z$. It follows from \cite[Example 7.10]{AB2} that 
\[ \Delta_3'=\{ f_{r,s}\in\Z[x] \colon D(f_{r,s})>0, \, r\neq s \},  \]
where $D(f_{r,s})=r^2 s^2 - 4r^3 - 4s^3 + 18rs - 27$. Since $f_{r,s}\in \Delta_3'$, no root of $f_{r,s}$ is equal to 1. Then there are two cases: (a) one root is greater than 1 and the other two roots are in the interval $(0,1)$, or (b) two roots are greater than 1 and the other root is in $(0,1)$. In case (a), we are in the conditions of Proposition \ref{prop: tracy}$\ri$, and therefore $p$ has full rank. In case (b), the polynomial $p^*$ satisfies the conditions in Proposition \ref{prop: tracy}$\ri$, so that $p^*$ has full rank. But, this implies that $p$ has full rank as well, according to Remark \ref{rem:full rank}. That is, any polynomial $f_{r,s}$ in $\Delta_3'$ has full rank. Consequently, the Betti numbers of $M_{f_{r,s}}$ are $b_k=1$ for $k=0,1,3,4$ and $b_2=0$.
\end{example}

\begin{remark}
Let $p\in\Delta_n'$ satisfy the full rank condition. As a direct consequence of Proposition \ref{prop: betti_base_case}, the Poincaré polynomial of the solvmanifold $M_p$ of dimension $n+1$ is
\[
P_{M_p}(t)=1+t+t^n+t^{n+1}=(1+t)(1+t^n).
\]
\end{remark}

\begin{remark}
Assume that $n$ is odd, so that the dimension of $M_{p} = \Gamma \backslash G_p$ is even. Then $M_p$ does not admit any symplectic form (invariant or not) since $b_2(M_p)=0$, and this contradicts the well-known fact that any symplectic form $\omega$ in $M_p$ defines a non-trivial cohomology class in $H^2_{dR}(M_p)$.
 
Now, recall that a \textit{complex symplectic} structure on a manifold $N$ is a pair $(J,\omega_\C)$ of a complex structure $J$ on $N$ and a closed non-degenerate 2-form $\omega_\C\in \Omega^2(N,\C)$ of bidegree $(2,0)$ with respect to $J$. It follows that the real dimension of $N$ is a multiple of 4. If, in our setting, we assume $n=4m-1$ ($m\geq 1$), so that $M_p = \Gamma \backslash G_p$ has dimension $4m$, then $M_p$ does not admit any complex symplectic structure, invariant or not. Indeed, it was proved in \cite[Corollary 6.2]{BFLT} that a compact manifold admitting a complex symplectic structure has $b_2\geq 2$, but for the basic solvmanifold we have $b_2(M_p)=0$.
\end{remark}

\subsection{The quasi full rank case}

\begin{proposition} \label{prop: betti_reciprocal}
Let $p \in \Delta'_{n}$ with $n = 2m$ be a polynomial satisfying the quasi full rank condition. Then, for the associated basic $(n+1)$-dimensional solvmanifold, the Betti numbers are given by
\[
b_{2j}=b_{2j+1}=\binom{m}{j},
\qquad 0\le j\le m.
\]
\end{proposition}

\begin{proof}
By Lemma \ref{lemma: betti_combinatorial} (with $d = 2m$), the Betti numbers are determined by the integers $\ell_k$. Since each index in $\{1, \dots, 2m\}$ corresponds to a unique parameter in 
\[
\{s_1,-s_1,\ldots,s_m,-s_m\},
\]
(with index $2i-1$ corresponding to $s_i$ and index $2i$ to $-s_i$, for $i = 1, \dots, m$), counting index subsets of size $k$ with vanishing parameter sum is equivalent to counting subsets of parameters of size $k$ adding up to zero.

By the quasi full rank condition \eqref{eq: QFR}, such a sum vanishes if and only if it consists of complete pairs $\{s_i, -s_i\}$. It follows that $\ell_k = 0$ whenever $k$ is odd. Moreover, if $k = 2j$, a vanishing sum arises precisely by choosing $j$ complete pairs among the $m$ available 
ones. Therefore,
\[
    \ell_{2j} = \binom{m}{j}, \quad 0 \leq j \leq m.
\]
Applying \eqref{eq: betti_ell}, namely $b_k=\ell_k+\ell_{k-1}$, we obtain
\[
b_{2j}=\ell_{2j}+\ell_{2j-1}=\binom{m}{j} \qquad \text{and} \qquad b_{2j+1}=\ell_{2j+1}+\ell_{2j}=\binom{m}{j}, \qquad \text{for all} \quad 0\le j\le m. \qedhere
\]
\end{proof}

\begin{example} \label{ex:pol_reducible}
    Using Lemma \ref{thm:disjoint_extensions} or Corollary \ref{lemma: coprime} in the Appendix, we can easily construct basic solvmanifolds of any even dimension $2m$ that satisfy the hypotheses of Proposition \ref{prop: betti_reciprocal}.
    
    For $m=2$ (dimension $5$), it suffices to choose integers $r, s \ge 3$ such that $r^2-4$ and $s^2-4$ have different square-free parts. For instance, taking $r=3$ and $s=4$, we consider the polynomials $h_3(x) = x^2 - 3x + 1$ and $h_4(x) = x^2 - 4x + 1$. The explicit roots are $\frac{3 \pm \sqrt{5}}{2}$ and $2 \pm \sqrt{3}$, respectively. The resulting self-reciprocal polynomial is:
    \[
        p(x) = h_3(x)h_4(x) = x^4 - 7x^3 + 14x^2 - 7x + 1.
    \]
    This polynomial yields roots in $\Q(\sqrt{5})$ and $\Q(\sqrt{12}) \cong \Q(\sqrt{3})$. Since these quadratic fields are linearly disjoint, there are no extra relations among the roots, and the Betti numbers of the corresponding solvmanifold are completely determined: $b_0=b_1=1, b_2=b_3=2, b_4=b_5=1$.
    
    For $m=3$ (dimension $7$), we can extend this by adding a third polynomial, say $h_5(x) = x^2 - 5x + 1$, whose roots are $\frac{5 \pm \sqrt{21}}{2}$. Multiplying the three factors yields the following self-reciprocal polynomial of degree $6$:
    \[
        p(x) = h_3(x)h_4(x)h_5(x) = x^6 - 12x^5 + 50x^4 - 84x^3 + 50x^2 - 12x + 1
    \]
    The discriminants of the quadratic factors are $5, 12, 21$, whose square-free parts are $5, 3, 21$. Since no nontrivial product of these three numbers is a square, the three quadratic fields are linearly disjoint. By Proposition \ref{prop: betti_reciprocal}, the Betti numbers are given by the binomial coefficients $\binom{3}{j}$, specifically: $b_0=b_1=1$, $b_2=b_3=3$, $b_4=b_5=3$, $b_6=b_7=1$.
\end{example}

Note that the self-reciprocal polynomials appearing in Example \ref{ex:pol_reducible} are all reducible over $\Q$. We will show next that we can choose them to be irreducible over $\Q$; for this, we will extend the definition of the polynomials $h_m$ in Example \ref{ex: ejemplos basicos} by allowing $m$ to be a real number with $m>2$.

\begin{example}
    Consider $p(x) = h_r(x)h_s(x)$ with $r, s \notin \Q$ and $r>2, s>2$. Expanding, we obtain
    \[
    p(x) = x^4 - (r+s)x^3 + (rs+2)x^2 - (r+s)x + 1.
    \]
    Therefore, $p(x)$ has integer coefficients if and only if $r+s = A \in \Z$ and $rs = B \in \Z$. In this case, $r$ and $s$ are roots of the quadratic equation $y^2 - Ay + B = 0$.
    
    For example, choosing $r = 4 + \sqrt{2}$ and $s = 4 - \sqrt{2}$ yields $A=8$ and $B=14$. Since both $r$ and $s$ are greater than $2$, all four roots of the resulting self-reciprocal polynomial
    \[
    p(x) = x^4 - 8x^3 + 16x^2 - 8x + 1
    \]
    are real, positive, and distinct. Moreover, this polynomial is irreducible over $\Q$, belongs to $\Delta'_4$, and its logarithmic parameters once again partition into exactly two zero-sum pairs with no additive relations between them, yielding exactly the same Betti numbers as in the reducible case, that is, $b_0=b_1=1$, $b_2=b_3=2$, $b_4=b_5=1$.
\end{example}

\begin{remark}
Let $p\in\Delta_n'$ be a polynomial satisfying the quasi full rank condition. Then the Betti numbers described in Proposition \ref{prop: betti_reciprocal} determine the Poincaré polynomial of the associated solvmanifold $M_p$ of dimension $2n+1=4m+1$:
\[
P_{M_p}(t)=\sum_{k=0}^{2n+1} b_k t^k
      =(1+t)\sum_{j=0}^{m}\binom{m}{j} t^{2j}.
\]
Using the binomial formula, this can be written as
\[
P_{M_p}(t)=(1+t)(1+t^2)^m.
\]
\end{remark}

\section{De Rham cohomology of the complex solvmanifolds}

We now turn our attention to the complex solvmanifolds $M_p^c = \Gamma^c \backslash G_p^c$. As before, Hattori's theorem implies that $H^*_{dR}(M_p^c) \cong H^*(\g_p^c)$. 

Recall that the $(2n+2)$-dimensional Lie algebra is defined as $\g_p^c = \R e_0 \ltimes_{A_p^c} \R^{2n+1}$, where the matrix $A_p^c = 0_1 \oplus A_p^{\oplus 2}$. The presence of the $0_1$ block indicates that there is a central element $f \in \R^{2n+1}$ which commutes with $e_0$. Consequently, the Lie algebra decomposes as a direct sum $\g_p^c = \tilde{\g}_p \times \R f$, where $\tilde{\g}_p$ is a $(2n+1)$-dimensional almost abelian Lie algebra. 

The structure of $\tilde{\g}_p$ is entirely determined by the $2n$ eigenvalues of $A_p^{\oplus 2}$. If $\{s_1, \dots, s_n\}$ are the logarithmic parameters associated with the roots of $p(x)$, then each parameter appears twice. Thus, the corresponding multiset of parameters is
\[
S=\{s_1,\dots,s_n,s_1,\dots,s_n\}.
\]

To compute the Betti numbers of $\tilde{\g}_p$, we rely on Lemma \ref{lemma: betti_combinatorial}. In the present situation, $\ell_k$ counts subsets $I\subset \{1,\dots,2n\}$ of cardinality $k$ whose associated parameters have vanishing sum. Identifying the indices $i$ and $n+i$ with the first and second occurrence of $s_i$, respectively, yields a bijection between subsets $I\subset\{1,\ldots,2n\}$ of cardinality $k$ and sub-multisets of $S$ of cardinality $k$. Moreover, under this bijection, the sum of the parameters indexed by $I$ coincides with the sum of the elements of the corresponding sub-multiset of $S$. Therefore, the quantities $\ell_k$ may equivalently be computed by counting sub-multisets of $S$ of cardinality $k$ whose elements sum to zero, and we will freely use this identification throughout this section.

Again, we will consider two cases: (i) full rank polynomials; (ii) quasi full rank polynomials.



\subsection{The full rank case}

\begin{lemma} \label{lemma: l_k_full_rank}
    If the polynomial $p \in \Delta'_n$ has full rank, then the values of $\ell_k$ for the multiset $S$ are given by:
    \begin{equation*}
        \ell_k  = \begin{cases}
            1 & \text{if } k = 0 \text{ or } k = 2n, \\
            2^n & \text{if } k = n, \\
            0 & \text{otherwise}.
        \end{cases}
    \end{equation*}
\end{lemma}

\begin{proof}
Let $J \subset S$ be a sub-multiset of size $k$. For each index $i \in \{1, \dots, n\}$, let $c_i \in \{0, 1, 2\}$ denote the number of copies of the parameter $s_i$ present in $J$. The condition that the elements of $J$ sum to zero can be written as
\[
    \sum_{i=1}^n c_i s_i = 0.
\]

By the full rank condition \eqref{eq:k_i iguales}, this identity holds if and only if all the coefficients $c_i$ are equal to some constant $c$:
\[
    c_1 = c_2 = \dots = c_n =: c.
\]
    
Since $c_i \in \{0, 1, 2\}$, there are only three possible cases for $c$.
    
If $c=0$, then $J=\emptyset$, so $k=0$ and $\ell_0=1$. If $c=1$, exactly one copy of each $s_i$ is chosen, hence $k=n$; since $S$ contains two copies of each $s_i$, there are $2^n$ such choices, giving $\ell_n=2^n$. Finally, if $c=2$, all elements of $S$ are selected, so $k=2n$ and $\ell_{2n}=1$. Any other subset of size $k$ would require non-uniform coefficients $c_i$, which is impossible under the full rank hypothesis. Thus, $\ell_k  = 0$ for all other $k$.
\end{proof}

With the values of $\ell_k $ explicitly determined, we can easily compute the cohomology of the subalgebra $\tilde{\g}_p$.

\begin{lemma} \label{lemma: betti_tilde_g}
    Let $p\in \Delta_n'$ have full rank. The Betti numbers $b_k(\tilde{\g}_p)$ of the $(2n+1)$-dimensional Lie algebra $\tilde{\g}_p$ are given by:
    \begin{equation*}
        b_k(\tilde{\g}_p) = \begin{cases}
            1 & \text{if } k \in \{0, 1, 2n, 2n+1\}, \\
            2^n & \text{if } k \in \{n, n+1\}, \\
            0 & \text{if } 1 < k < n \text{ or } n+1 < k < 2n.
        \end{cases}
    \end{equation*}
\end{lemma}

\begin{proof} 
By Lemma \ref{lemma: betti_combinatorial},
\[
b_k(\tilde{\g}_p)=\ell_k+\ell_{k-1}
\]
for $0\leq k\leq 2n+1$. Recall that $\ell_k=0$ whenever $k<0$ or $k>2n$. Combining this with Lemma \ref{lemma: l_k_full_rank}, we obtain
\[
b_0=b_1=b_{2n}=b_{2n+1}=1,
\qquad
b_n=b_{n+1}=2^n,
\]
while $b_k=0$ for all remaining degrees.
\end{proof}

We can now compute the Betti numbers of the full $(2n+2)$-dimensional Lie algebra $\g_p^c$. Since $\g_p^c=\tilde{\g}_p \times \R f$ and $\R f$ is a one-dimensional abelian Lie algebra, its only non-trivial cohomology groups are $H^0(\R f)\cong \R$ and $H^1(\R f)\cong \R$. The Künneth formula therefore yields $H^k(\g_p^c)\cong H^k(\tilde{\g}_p)\oplus H^{k-1}(\tilde{\g}_p)$. Taking dimensions, we obtain
\begin{equation}\label{eq: kunneth_complex}
b_k(\g_p^c)=b_k(\tilde{\g}_p)+b_{k-1}(\tilde{\g}_p).
\end{equation}

The following theorem gives the Betti numbers of the complex solvmanifolds under the full rank assumption.

\begin{theorem} \label{thm: cohomology_complex}
    Let $M_p^c = \Gamma^c \backslash G_p^c$ be the $(2n+2)$-dimensional complex solvmanifold associated with a polynomial $p \in \Delta'_n$ satisfying the full rank condition. The Betti numbers $b_k = b_k(M_p^c)$ are given as follows:
    
    \begin{itemize}
        \item \textbf{General case ($n \geq 3$):}
        \begin{equation*}
            b_k = \begin{cases}
                1 & \text{if } k \in \{0, 2, 2n, 2n+2\}, \\
                2 & \text{if } k \in \{1, 2n+1\}, \\
                2^n & \text{if } k \in \{n, n+2\}, \\
                2^{n+1} & \text{if } k = n+1, \\
                0 & \text{otherwise}.
            \end{cases}
        \end{equation*}
        
        \item \textbf{Particular case ($n = 2$):}
        The Betti numbers are
        \[
        (1, 2, 5, 8, 5, 2, 1).
        \]
    \end{itemize}
\end{theorem}

\begin{proof}
Using \eqref{eq: kunneth_complex} and Lemma~\ref{lemma: betti_tilde_g}, we obtain
\[
b_0=b_0(\tilde{\g}_p)=1,\qquad
b_1=b_1(\tilde{\g}_p)+b_0(\tilde{\g}_p)=2,
\]
and
\[
b_2=b_2(\tilde{\g}_p)+b_1(\tilde{\g}_p).
\]

At this point, the case $n=2$ must be treated separately. Since $b_1(\tilde{\g}_p)=1$, if $n\geq 3$ then $b_2(\tilde{\g}_p)=0$, and hence $b_2=1$. On the other hand, when $n=2$, the degree $k = 2$ coincides with the middle degree $n$. According to Lemma \ref{lemma: betti_tilde_g}, $b_n(\tilde{\g}_p)=2^n=4$, and therefore $b_2=4+1=5$, due to the overlap between the two contributions.

For the middle degrees $k\in\{n,n+1,n+2\}$ with $n\geq 3$, we have
\[
\begin{aligned}
b_n &= b_n(\tilde{\g}_p)+b_{n-1}(\tilde{\g}_p)
     = 2^n+0
     = 2^n,\\
b_{n+1} &= b_{n+1}(\tilde{\g}_p)+b_n(\tilde{\g}_p)
         = 2^n+2^n
         = 2^{n+1},\\
b_{n+2} &= b_{n+2}(\tilde{\g}_p)+b_{n+1}(\tilde{\g}_p)
         = 0+2^n
         = 2^n.
\end{aligned}
\]

For $n=2$, we similarly obtain
\[
b_3=b_3(\tilde{\g}_p)+b_2(\tilde{\g}_p)
   =4+4
   =8.
\]

Finally, the remaining values, corresponding to $n+2<k\leq 2n+2$, follow by Poincaré duality on $M_p^c$.
\end{proof}

\begin{remark}
It follows from Theorem \ref{thm: cohomology_complex} that, under the same hypotheses and assuming $n\geq 3$, the Poincaré polynomial of the $(2n+2)$-dimensional complex solvmanifold $M_p^c$ is
\[
P_{M_p^c}(t) = 1 + 2t + t^2 + 2^n t^n + 2^{n+1} t^{n+1} + 2^n t^{n+2} + t^{2n} + 2t^{2n+1} + t^{2n+2}.
\]
Equivalently,
\[
P_{M_p^c}(t) = (1 + t)^2 (1 + 2 ^n t^n + t^{2n}).
\]
\end{remark}

\begin{remark}
    When $n=2$, the solvmanifold $M_p^c$  corresponds to the completely solvable Nakamura manifold introduced by I. Nakamura in \cite[page 90]{Na}. The Betti numbers of this manifold can be found, for instance, in \cite[Table 3.7]{Ang}, and coincide with the values obtained in Theorem \ref{thm: cohomology_complex}. In Section \ref{sec: dolbeault} we will discuss the Dolbeault cohomology of the manifolds $M_p^c$ in any dimension.
\end{remark}

\subsection{The quasi full rank case}

Let $p \in \Delta'_n$ be a self-reciprocal polynomial of even degree $n = 2m$. As mentioned previously, its roots naturally appear in reciprocal pairs
\[
\{r_1,r_1^{-1},\dots,r_m,r_m^{-1}\},
\]
so the corresponding logarithmic parameters are
\[
\{\pm s_1,\dots,\pm s_m\},
\qquad s_i=\log r_i.
\]

For the associated Lie algebra $\tilde{\g}_p$ of dimension $2n+1$, the duplicated multiset of parameters with $2n = 4m$ elements is:
\[
    S = \{s_1, s_1, -s_1, -s_1, s_2, s_2, -s_2, -s_2, \dots, s_m, s_m, -s_m, -s_m\}.
\]

To compute the Betti numbers of $\tilde{\g}_p$, we need to determine the numbers $\ell_k$. Using the convention introduced above, and following the ordering of $S$ (so that indices $4i-3, 4i-2$ correspond to the two occurrences of $s_i$ and indices $4i-1, 4i$ to those of $-s_i$), we regard $\ell_k$ as the number of sub-multisets of $S$ of cardinality $k$ whose elements sum to zero. Under the quasi full rank condition, this combinatorial problem can be solved using generating functions.

\begin{theorem} \label{thm: qfr_generating_function}
    Let $S$ be the multiset of parameters associated with a polynomial $p \in \Delta'_n$ satisfying the quasi full rank condition, with $n = 2m$. Then $\ell_k $ is given by the coefficient of $t^k$ in the following polynomial:
    \[
        L(t) = (1 + 4t^2 + t^4)^m.
    \]
\end{theorem}

\begin{proof}
    By the quasi full rank hypothesis, there are no non-trivial linear relations involving parameters with different indices. Hence, the multiset $S$ decomposes into $m$ independent blocks
    \[
    B_i=\{s_i,s_i,-s_i,-s_i\},
    \qquad i=1,\dots,m.
    \]

    Any sub-multiset $A \subseteq S$ can be uniquely written as a disjoint union $A = \bigcup_{i=1}^m A_i$, where $A_i = A \cap B_i$. The sum of the elements in $A$ vanishes if and only if the sum of the elements in each component $A_i$ vanishes.

    Let $u_j$ denote the number of zero-sum sub-multisets of size $j$ contained in a single block $B_i$. Such a sub-multiset must contain the same number, say $l$, of copies of $s_i$ and $-s_i$. Since each sign appears twice, the only possibilities are
    \[
    u_0=\binom{2}{0}\binom{2}{0}=1,\qquad
    u_2=\binom{2}{1}\binom{2}{1}=4,\qquad
    u_4=\binom{2}{2}\binom{2}{2}=1,
    \]
    while $u_j=0$ for all other values of $j$.

    Therefore, the generating polynomial associated with a single block is
    \[
    L_{\mathrm{block}}(t) = \sum_{j=0}^4 u_j t^j = 1+4t^2+t^4.
    \]

    To construct a zero-sum sub-multiset of size $k$, we choose integers
    \[
    k_1 + k_2 \cdots + k_m = k
    \]
    and, for each block $B_i$, a valid zero-sum configuration of size $k_i$. Thus
    \[
    \ell_k = \sum_{k_1+\cdots+k_m=k} u_{k_1}\cdots u_{k_m}.
    \]

    By the definition of polynomial multiplication, this is precisely the coefficient of $t^k$ in
    \[
    L(t) = L_{\mathrm{block}}(t)^m = (1+4t^2+t^4)^m. \qedhere
    \]
\end{proof}

The cohomology of our Lie algebras can be conveniently encoded by their Poincaré polynomials, defined by
\[
P_{\g}(t)=\sum_{k\geq 0} b_k(\g)\, t^k.
\]
The generating function obtained in Theorem \ref{thm: qfr_generating_function},
\[
L(t)=(1+4t^2+t^4)^m=\sum_{k\geq 0}\ell_k t^k,
\]
allows us to derive explicit expressions for the Poincaré polynomials of both $\tilde{\g}_p$ and $\g_p^c$, and consequently for the associated complex solvmanifold $M_p^c$. Observe that $L(t)$ contains only even powers of $t$, and therefore
\[
\ell_k=0
\qquad \text{whenever } k \text{ is odd}.
\]

\begin{corollary} \label{cor: poincare_polynomials}
Let $p \in \Delta'_n$ be a polynomial satisfying the quasi full rank condition, with $n=2m$. The Poincaré polynomial of the associated complex solvmanifold $M_p^c$ is given by
\begin{equation*}
    P_{M_p^c}(t) = (1+t)^2(1 + 4t^2 + t^4)^m. \label{eq: poincare_g^c}
\end{equation*}
\end{corollary}

\begin{proof}
By Lemma \ref{lemma: betti_combinatorial}, the Betti numbers of $\tilde{\g}_p$ satisfy $b_k(\tilde{\g}_p)=\ell_k+\ell_{k-1}$. Hence, the Poincaré polynomial of $\tilde{\g}_p$ is
\[
P_{\tilde{\g}_p}(t) = (1+t)L(t).
\]

Since $\g_p^c$ is the complete complex Lie algebra associated with $\tilde{\g}_p$, the Künneth formula \eqref{eq: kunneth_complex} implies that
\[
P_{\g_p^c}(t)=(1+t)P_{\tilde{\g}_p}(t).
\]
Using this identity together with Theorem \ref{thm: qfr_generating_function}, we obtain
\[
P_{M_p^c}(t) = P_{\g_p^c}(t) = (1+t)^2(1+4t^2+t^4)^m. \qedhere
\]
\end{proof}

\begin{remark} \label{rem: l_k}
Corollary \ref{cor: poincare_polynomials} shows that the Betti numbers of the complex solvmanifold $M_p^c$ are completely determined by the generating polynomial
\[
L(t)=(1+4t^2+t^4)^m.
\]
Indeed,
\[
P_{M_p^c}(t) = (1+t)^2L(t) = (1+2t+t^2)\sum_{j\geq 0}\ell_j t^j,
\]
and therefore the coefficient of $t^k$ satisfies
\[
b_k(M_p^c)=\ell_k+2\ell_{k-1}+\ell_{k-2}.
\]
Since $\ell_j=0$ for odd $j$, it follows that
\[
b_{2k}= \ell_{2k}+\ell_{2k-2},
\qquad
b_{2k+1}=2\ell_{2k}.
\]
Although explicit formulas for the coefficients $\ell_k$ can be obtained from the expansion of the polynomial $L(t)$, the resulting expressions are not particularly enlightening.
\end{remark}

\begin{example}
    Consider a quasi full rank polynomial of degree $n = 6$ ($m = 3$). The generating polynomial is $L(t) = (1 + 4t^2 + t^4)^3$. Expanding this expression, we obtain:
    \[
    L(t) = 1 + 12t^2 + 51t^4 + 88t^6 + 51t^8 + 12t^{10} + t^{12}
    \]
    Reading the coefficients directly yields the values for $\ell_{2k}$. For instance, $\ell_4 = 51$ and $\ell_6 = 88$. Applying the formulas $b_{2k} = \ell_{2k} + \ell_{2k-2}$ and $b_{2k+1} = 2\ell_{2k}$ derived in Remark \ref{rem: l_k}, we can directly compute the Betti numbers for the 14-dimensional complex solvmanifold $M_p^c$:
    \[
    (b_0, \dots, b_{14}) = (1, 2, 13, 24, 63, 102, 139, 176, 139, 102, 63, 24, 13, 2, 1).
    \]
\end{example}

\begin{remark}
    We point out that, even though the complex solvmanifolds $M^c_p$ do not admit any Kähler metric since they are completely solvable (see \cite{Hasegawa}), their odd-indexed Betti numbers $b_{2k+1}$ are even in both the full-rank and the quasi-full-rank settings, just like the Betti numbers of a compact Kähler manifold.
\end{remark}

\section{De Rham cohomology of the hypercomplex solvmanifolds}

Let $M_p^{h} = \Gamma^{h} \backslash G_p^{h}$ be the hypercomplex solvmanifold associated to $p \in \Delta'_n$. By Hattori's theorem, we can compute its de Rham cohomology by calculating the Chevalley-Eilenberg cohomology of its Lie algebra, i.e., $H^*_{dR}(M_p^{h}) \cong H^*(\g_p^{h})$.

Recall from Section \ref{sec: construction} that the $(4n+4)$-dimensional hypercomplex Lie algebra is given by $\g_p^{h} = \R e_0 \ltimes_{A_p^{h}} \R^{4n+3}$, where the action is determined by the block diagonal matrix $A_p^{h} = 0_3 \oplus A_p^{\oplus 4}$. We can decompose the abelian ideal $\R^{4n+3}$ into a direct sum of two vector spaces:
\begin{equation}\label{eq: V_0_mas_V_1}
    \R^{4n+3} = V_0 \oplus V_1,
\end{equation}
where $V_0 = \operatorname{span}\{e_1, e_2, e_3\} \cong \R^3$ corresponds to the $0_3$ block, and $V_1 \cong \R^{4n}$ corresponds to the subspace where the action is given by the block $A_p^{\oplus 4}$.

Since $A_p^{h}|_{V_0} = 0$, the elements of $V_0$ commute with $e_0$ (i.e., $[e_0, v] = 0$ for all $v \in V_0$). Because $\R^{4n+3}$ is an abelian ideal, $V_0$ also commutes with the rest of the basis elements. This implies that $V_0$ lies in the center of $\g_p^{h}$. As a consequence, the semidirect product naturally splits into a direct sum of Lie algebras:
\begin{equation*}
    \g_p^{h} = \bar{\g}_p \times V_0 \cong \bar{\g}_p \times \R^3,
\end{equation*}
where $\bar{\g}_p = \R e_0 \ltimes_{A_p^{\oplus 4}} V_1$ is a $(4n+1)$-dimensional almost abelian Lie algebra.

Based on this direct sum decomposition, we can apply the Künneth formula to compute the cohomology of $\g_p^{h}$:
\[
    H^k(\g_p^{h}) \cong \bigoplus_{j=0}^k \left( H^{k-j}(\bar{\g}_p) \otimes H^j(\R^3) \right).
\]

Since $\R^3$ is a 3-dimensional abelian Lie algebra, its Betti numbers are simply given by the binomial coefficients $b_j(\R^3) = \binom{3}{j}$. Thus, the nonzero Betti numbers for the $\R^3$ factor are $b_0=1$, $b_1=3$, $b_2=3$, and $b_3=1$. This allows us to express the Betti numbers of the full Lie algebra explicitly in terms of the Betti numbers of the subalgebra $\bar{\g}_p$:
\begin{equation} \label{eq: kunneth_hypercomplex}
    b_k(\g_p^{h}) = b_k(\bar{\g}_p) + 3 b_{k-1}(\bar{\g}_p) + 3 b_{k-2}(\bar{\g}_p) + b_{k-3}(\bar{\g}_p).
\end{equation}

To make use of this formula, we must first compute $b_k(\bar{\g}_p)$. The structure of $\bar{\g}_p$ is entirely determined by the eigenvalues of $A_p^{\oplus 4}$. Letting $\{s_1, \dots, s_n\}$ be the logarithmic parameters associated with $p \in \Delta_n'$, the multiset of eigenvalues $S$ for this action consists of exactly four copies of each parameter:

\[
S= \{s_1,\dots,s_n,\, s_1,\dots,s_n,\, s_1,\dots,s_n,\, s_1,\dots,s_n\}.
\]


The total size of the multiset $S$ is $4n$. 
As in the previous section, to apply Lemma \ref{lemma: betti_combinatorial} we identify the indices $i,\; n+i,\; 2n+i,\; 3n+i$ with the four occurrences of $s_i$, respectively. In this way, subsets $I\subset\{1,\ldots,4n\}$ of cardinality $k$ correspond bijectively to sub-multisets of $S$ of cardinality $k$, preserving the associated sums. Hence, $\ell_k$ may be computed by counting sub-multisets of $S$ whose elements sum to zero. In what follows, we will freely use this interpretation of the quantities $\ell_k$.

\subsection{The full rank condition}

We assume here that the polynomial $p$ satisfies the full rank condition. The numbers $\ell_k$ are explicitly determined as follows:

\begin{lemma} \label{lemma: l_k_full_rank_hyp}
    If $p \in \Delta'_n$ satisfies the full rank condition, then the values of $\ell_k $ associated with the multiset $S$ are given by:

    \begin{equation*}
        \ell_k  = \begin{cases}
            1 & \text{if } k = 0 \text{ or } k = 4n, \\
            4^n & \text{if } k = n \text{ or } k = 3n, \\
            6^n & \text{if } k = 2n, \\
            0 & \text{otherwise}.
        \end{cases}
    \end{equation*}
\end{lemma}

\begin{proof}
Consider a sub-multiset $J \subset S$ of size $k$ whose elements sum to zero. For each $i \in \{1, \dots, n\}$, let $c_i \in \{0,1,2,3,4\}$ denote the number of copies of the parameter $s_i$ contained in $J$. The zero-sum condition can then be written as
\[
\sum_{i=1}^n c_i s_i = 0.
\]

By the full rank condition \eqref{eq:k_i iguales}, this identity holds if and only if all coefficients are equal, i.e.
\[
c_1 = c_2 = \cdots = c_n =: c, \qquad c \in \{0,1,2,3,4\}.
\]

Consequently, any valid sub-multiset must contain exactly $c$ copies of every parameter, implying that its total size must be of the form $k = c n$. For any other subset size $k \notin \{0, n, 2n, 3n, 4n\}$, the uniform coefficient condition cannot be met, which immediately yields $\ell_k = 0$.

When $k = c n$, the number of ways to choose $c$ copies from the four available for each of the $n$ parameters is given by the product of binomial coefficients $\binom{4}{c}^n$. Evaluating this expression for each possible value of $c$ directly produces 
\[
\ell_0 = 1, \qquad \ell_n = 4^n, \qquad \ell_{2n} = 6^n, \qquad \ell_{3n} = 4^n, \qquad \ell_{4n} = 1. \qedhere
\]
\end{proof}

Using these values, we can immediately compute the Betti numbers of the subalgebra $\bar{\g}_p$.

\begin{lemma} \label{lemma: betti_tilde_g_hyp}
    Let $p \in \Delta'_n$ be a polynomial satisfying the full rank condition. The Betti numbers $b_k(\bar{\g}_p)$ of the $(4n+1)$-dimensional Lie algebra $\bar{\g}_p$ are given by:
    
    \begin{equation*}
        b_k(\bar{\g}_p) = \begin{cases}
            1 & \text{if } k \in \{0, 1, 4n, 4n+1\}, \\
            4^n & \text{if } k \in \{n, n+1, 3n, 3n+1\}, \\
            6^n & \text{if } k \in \{2n, 2n+1\}, \\
            0 & \text{otherwise}.
        \end{cases}
    \end{equation*}
\end{lemma}

\begin{proof}
As established previously, the Betti numbers of the almost abelian Lie algebra $\bar{\g}_p$ are given by the relation \eqref{eq: betti_ell}: $b_k(\bar{\g}_p) = \ell_k  + \ell_{k-1}$. By Lemma \ref{lemma: l_k_full_rank_hyp}, the coefficient $\ell_j$ is nonzero if and only if $j = c n$ for some constant $c \in \{0, 1, 2, 3, 4\}$. Consequently, the sum $\ell_k  + \ell_{k-1}$ can only be nonzero when the degree $k$ takes the form $c n$ or $c n + 1$. Substituting the corresponding values $\ell_{cn} = \binom{4}{c}^n$ directly into the formula yields the stated Betti numbers for these specific degrees. For all other intermediate degrees, both terms in the sum vanish, resulting in trivial cohomology.
\end{proof}

\begin{theorem} \label{thm: cohomology_hyp}
    Let $p \in \Delta'_n$ be a polynomial of degree $n \ge 5$ satisfying the full rank condition. The Betti numbers $b_k = b_k(M_p^{h})$ of the associated $(4n+4)$-dimensional hypercomplex solvmanifold $M_p^{h}$ are given by:
    \begin{equation*}
        b_k = \begin{cases}
            \binom{4}{k - c n} \binom{4}{c}^n & \text{if } c n \le k \le c n + 4 \text{ for some } c \in \{0, 1, 2, 3, 4\}, \\
            0 & \text{otherwise}.
        \end{cases}
    \end{equation*}
\end{theorem}

\begin{proof}
    Recall that the Betti numbers of the algebra $\g_p^h$ are given by \eqref{eq: kunneth_hypercomplex}, that is:
    \[
        b_k(\g_p^{h}) = b_k(\bar{\g}_p) + 3b_{k-1}(\bar{\g}_p) + 3b_{k-2}(\bar{\g}_p) + b_{k-3}(\bar{\g}_p).
    \]
    
    From Lemma \ref{lemma: betti_tilde_g_hyp}, the nonzero Betti numbers of $\bar{\g}_p$ only occur at degrees $cn$ and $cn+1$, taking the value $\binom{4}{c}^n$ for each $c \in \{0, 1, 2, 3, 4\}$. Since $n \ge 5$, the five intervals $[cn, cn+4]$ are disjoint, so there is no overlap of nonzero terms in \eqref{eq: kunneth_hypercomplex} between different values of $c$.

    Fixing a constant $c \in \{0, 1, 2, 3, 4\}$, we compute the Betti numbers for $k = cn+i$ with $i \in \{0, 1, 2, 3, 4\}$ by evaluating the sum directly:
    \begin{align*}
        b_{cn} & = b_{cn}(\bar{\g}_p) = \binom{4}{c}^n = \binom{4}{0} \binom{4}{c}^n, \\
        b_{cn+1} & = b_{cn+1}(\bar{\g}_p) + 3b_{cn}(\bar{\g}_p) = 4\binom{4}{c}^n = \binom{4}{1} \binom{4}{c}^n, \\
        b_{cn+2} & = 3b_{cn+1}(\bar{\g}_p) + 3b_{cn}(\bar{\g}_p) = 6\binom{4}{c}^n = \binom{4}{2} \binom{4}{c}^n, \\
        b_{cn+3} & = 3b_{cn+1}(\bar{\g}_p) + b_{cn}(\bar{\g}_p) = 4\binom{4}{c}^n = \binom{4}{3} \binom{4}{c}^n, \\
        b_{cn+4} & = b_{cn+1}(\bar{\g}_p) = \binom{4}{c}^n = \binom{4}{4} \binom{4}{c}^n
    \end{align*}

    This shows that 
    \[
    b_{cn+i}(\g_p^{h}) = \binom{4}{i}\binom{4}{c}^n, \qquad \forall i \in \{0, 1, 2, 3, 4\}.
    \]
    Substituting $i = k - c n$ produces the expression in the theorem. Since $b_k(\bar{\g}_p) = 0$ outside these intervals, the Künneth formula then implies that all other Betti numbers vanish.
\end{proof}

\begin{remark}
    For practical reference, evaluating the closed-form expression from Theorem \ref{thm: cohomology_hyp} for each $c \in \{0, 1, 2, 3, 4\}$ yields the following explicit values for the nonzero Betti numbers of $M_p^{h}$:
    \begin{equation*}
        b_k = \begin{cases}
            1 & \text{if } k \in \{0, 4, 4n, 4n+4\}, \\
            4 & \text{if } k \in \{1, 3, 4n+1, 4n+3\}, \\
            6 & \text{if } k \in \{2, 4n+2\}, \\
            4^n & \text{if } k \in \{n, n+4, 3n, 3n+4\}, \\
            4^{n+1} & \text{if } k \in \{n+1, n+3, 3n+1, 3n+3\}, \\
            6 \cdot 4^n & \text{if } k \in \{n+2, 3n+2\}, \\
            6^n & \text{if } k \in \{2n, 2n+4\}, \\
            4 \cdot 6^n & \text{if } k \in \{2n+1, 2n+3\}, \\
            6^{n+1} & \text{if } k = 2n+2, \\
            0 & \text{otherwise}.
        \end{cases}
    \end{equation*}
\end{remark}

Since Theorem \ref{thm: cohomology_hyp} is explicitly stated for $n \geq 5$ to prevent overlap between the nonzero Betti degree intervals, we now provide the complete computations for the lower-dimensional cases $n=2$, $n=3$ and $n=4$. This concludes the cohomological characterization of the hypercomplex solvmanifolds under the full rank assumption.

\begin{example}
By Lemma \ref{lemma: betti_tilde_g_hyp}, the non-zero Betti numbers of the subalgebra $\bar{\g}_p$ occur in consecutive pairs
\[
b_{cn}(\bar{\g}_p) = b_{cn+1}(\bar{\g}_p) = \binom{4}{c}^n,
\quad c \in \{0, \dots, 4\}.
\]
Using \eqref{eq: kunneth_hypercomplex}, we obtain the following Betti sequences.

\medskip
\noindent
\begin{align*}
n=2 \ (\dim M_p^h = 12) \quad & (1, 4, 22, 68, 133, 208, 248, 208, 133, 68, 22, 4, 1), \\
n=3 \ (\dim M_p^h = 16) \quad & (1, 4, 6, 68, 257, 384, 472, 928, 1296, 928, 472, 384, 257, 68, 6, 4, 1), \\
n=4 \ (\dim M_p^h = 20) \quad & (1, 4, 6, 4, 257, 1024, 1536, 1024, 1552, 5184, 7776, 5184, 1552, \\
& 1024, 1536, 1024, 257, 4, 6, 4, 1).
\end{align*}
\end{example}

As in the complex case, we can compactly encode the Betti numbers of the hypercomplex solvmanifolds using Poincaré polynomials. 
The relation \eqref{eq: betti_ell}, that is, $b_k(\bar{\g}_p) = \ell_k  + \ell_{k-1}$, allows us to factor the polynomial as follows: 
\begin{equation*}
    P_{\bar{\g}_p}(t) = (1+t) \sum_{k=0}^{4n} \ell_k  t^k.
\end{equation*}
Using the nonzero values of $\ell_k $ obtained in Lemma \ref{lemma: l_k_full_rank_hyp} for the full rank case, we get:
\begin{equation*}
    P_{\bar{\g}_p}(t) = (1+t) \left( 1 + 4^n t^n + 6^n t^{2n} + 4^n t^{3n} + t^{4n} \right).
\end{equation*}

To find the Poincaré polynomial for the full hypercomplex Lie algebra $\g_p^{h}$, we use the direct sum decomposition $\g_p^{h} \cong \bar{\g}_p \times \R^3$. In terms of generating polynomials, the Künneth formula translates to a simple multiplication by the Poincaré polynomial of the central ideal. Since the Poincaré polynomial of $\R^3$ is $(1+t)^3$, we have:
\begin{equation*}
    P_{\g_p^{h}}(t) = P_{\bar{\g}_p}(t) \cdot (1+t)^3.
\end{equation*}
Therefore, the Betti numbers of the hypercomplex solvmanifold $M_p^h$ under the full rank condition are generated by the following polynomial:
\begin{equation*}\label{eq: poincare-hyp}
    P_{M_p^{h}}(t) = (1+t)^4 \left( 1 + 4^n t^n + 6^n t^{2n} + 4^n t^{3n} + t^{4n} \right).
\end{equation*}

\subsection{The quasi full rank condition}
Let $p \in \Delta'_n$ be a self-reciprocal polynomial of even degree $n = 2m$. Recall that its logarithmic parameters occur in opposite pairs, which we denote by 
\[ \{\pm s_1, \pm s_2, \dots, \pm s_m\}. 
\]

Consider the hypercomplex Lie algebra $\g_p^{h}$ of dimension $4n+4 = 8m+4$. The action on the abelian ideal $\R^{4n+3} = V_0 \oplus V_1$ (see \eqref{eq: V_0_mas_V_1}) is determined by the matrix $A_p^{h}$, which contains four copies of each parameter. More precisely, the diagonal action on $V_1 \cong \R^{4n}$ is given, in a suitable basis, by 
\[ \bigoplus_{i=1}^m \operatorname{diag}(s_i, s_i, s_i, s_i, -s_i, -s_i, -s_i, -s_i). 
\] 

Accordingly, the associated multiset of parameters $S$ for the subalgebra $\bar{\g}_p$, of size $4n = 8m$, contains four copies of each parameter. Reordering its elements and grouping together the parameters corresponding to each pair $\{\pm s_i\}$, we may write
\[
S = \{\underbrace{s_1,\dots,s_1}_{4},
      \underbrace{-s_1,\dots,-s_1}_{4},
      \dots,
      \underbrace{s_m,\dots,s_m}_{4},
      \underbrace{-s_m,\dots,-s_m}_{4}\}.
\]
For each $i \in \{1,\dots,m\}$, let
\[
B_i=\{s_i,s_i,s_i,s_i,-s_i,-s_i,-s_i,-s_i\}
\]
denote the corresponding block of parameters. Then
\[
S=\bigsqcup_{i=1}^{m} B_i.
\]

Recall that, as explained above, we identify subsets of indices in
$\{1,\ldots,4n\}$ with the corresponding sub-multisets of $S$ following 
the ordering above (so that indices $8i-7, 8i-6, 8i-5, 8i-4$ correspond to the four occurrences of $s_i$ and indices $8i-3, 8i-2, 8i-1, 8i$ to those of $-s_i$). Under this identification, the quantities $\ell_k$ appearing in 
Lemma \ref{lemma: betti_combinatorial} may be computed by counting 
sub-multisets of $S$ of cardinality $k$ whose elements sum to zero.

Assume now that $p$ satisfies the quasi full rank condition. By definition, the only vanishing linear relations among the parameters are the trivial cancellations within each pair $\pm s_i$, namely, 
\[ 
s_i + (-s_i)=0, 
\] 
and no relations involving distinct indices occur. Consequently, the problem of computing the numbers $\ell_k$ reduces to an independent combinatorial analysis on each block $B_i$. 

A zero-sum sub-multiset of $B_i$ is obtained precisely by selecting the same number $l$ of copies of $s_i$ and $-s_i$. Since each sign appears with multiplicity four, we have $l \in \{0,1,2,3,4\}$. The number of such choices is 
\[ 
\binom{4}{l}\binom{4}{l} = \binom{4}{l}^{2}, 
\] 
and the resulting sub-multiset has cardinality $k=2l$. Therefore, the only nonzero values of $\ell_k$ for a single block occur for
\[ 
k\in\{0,2,4,6,8\}, 
\] 
with 
\[ 
\ell_0 = 1, \qquad
\ell_2 = 16, \qquad
\ell_4= 36, \qquad
\ell_6 = 16, \qquad
\ell_8 = 1. 
\] 

Equivalently, these coefficients are encoded in the generating polynomial for a single block:
\[ 
L_{\operatorname{block}}(t) = 1 + 16t^2 + 36t^4 + 16t^6 + t^8. 
\]

Since $S$ is the disjoint union of $m$ independent blocks, we obtain the following result, which is completely analogous to Theorem \ref{thm: qfr_generating_function} for the complex case.

\begin{theorem} \label{thm: qfr_generating_function_hyp}
Let $S$ be the multiset of parameters associated with a quasi full rank polynomial $p \in \Delta'_n$, with $n = 2m$ in the hypercomplex setting. Then $\ell_k $ is exactly the coefficient of $t^k$ in the expansion of the generating polynomial:
\[
L(t) = [L_{\operatorname{block}}(t)]^m = (1 + 16t^2 + 36t^4 + 16t^6 + t^8)^m.
\]
\end{theorem}

\begin{remark} 
Since the generating polynomial $L(t)$ contains only even powers of $t$, it follows immediately that $\ell_k = 0$ whenever $k$ is odd. 
\end{remark}

\medskip

In analogy with the complex case, the Poincaré polynomials of the corresponding Lie algebras are readily obtained from Theorem \ref{thm: qfr_generating_function_hyp}.

\begin{corollary} \label{cor: poincare_polynomials_hyp}
Let $p \in \Delta'_n$ be a polynomial of degree $n=2m$ satisfying the quasi full rank condition. The Poincaré polynomials for the associated hypercomplex solvmanifold $M_p^h$ is given by:
    \begin{equation*}
        P_{M_p^{h}}(t) = (1+t)^4 L(t),
    \end{equation*}
with $L(t)$ as in Theorem \ref{thm: qfr_generating_function_hyp}.
\end{corollary}

\begin{proof}
By Lemma \ref{lemma: betti_combinatorial}, the Betti numbers of the almost abelian subalgebra $\bar{\g}_p$ satisfy $b_k(\bar{\g}_p)=\ell_k+\ell_{k-1}$.
Therefore,
\[
P_{\bar{\g}_p}(t)=(1+t)L(t).
\]

Now, by the Künneth formula \eqref{eq: kunneth_hypercomplex}, we have
\[
P_{M_p^h} = P_{\g_p^{h}}(t) = (1+t)^3 P_{\bar{\g}_p}(t) = (1+t)^4L(t),
\]
where $(1+t)^3$ is the Poincaré polynomial associated with the central factor $\R^3$.
\end{proof}

\begin{remark}\label{rem: l_k hyper}
The previous corollary shows that the cohomology of the hypercomplex solvmanifold $M_p^h$ is entirely encoded by the polynomial $L(t)$ appearing in Theorem \ref{thm: qfr_generating_function_hyp}.

Indeed, expanding the identity
\[
P_{M_p^h}(t) = (1+t)^4L(t) = (1+4t+6t^2+4t^3+t^4)\sum_{j\geq 0}\ell_j t^j,
\]
we obtain
\[
b_k(M_p^h) = \ell_k + 4\ell_{k-1} + 6\ell_{k-2} + 4\ell_{k-3} + \ell_{k-4}.
\]

Since $L(t)$ only contains even powers of $t$, we have $\ell_j=0$ for every odd $j$. Consequently,
\[
b_{2k} = \ell_{2k} + 6\ell_{2k-2} + \ell_{2k-4},
\]
while
\begin{equation} \label{eq: betti_impar}
    b_{2k+1} = 4\ell_{2k} + 4\ell_{2k-2}.
\end{equation}

As in the complex case, one may derive explicit formulas for the coefficients $\ell_k$ by expanding the polynomial $L(t)$, but we omit them here since they do not seem to provide additional insight.
\end{remark}

\begin{example}
Let $p \in \Delta'_n$ be a polynomial satisfying the quasi full rank condition, with $n = 4$ ($m=2$). The associated hypercomplex solvmanifold $M_p^h$ then has real dimension $20$. We begin by expanding the generating polynomial:
\begin{align*}
    L(t) &= (1 + 16t^2 + 36t^4 + 16t^6 + t^8)^2 \\
    &= 1 + 32t^2 + 328t^4 + 1184t^6 + 1810t^8 + 1184t^{10} + 328t^{12} + 32t^{14} + t^{16}.
\end{align*}
Hence,
\[
\ell_0 = 1, \quad \ell_2 = 32, \quad \ell_4 = 328, \quad \ell_6 = 1184, \quad \ell_8 = 1810,
\]
with the remaining coefficients determined by symmetry. Using the identities obtained in the Remark \ref{rem: l_k hyper},
\[
b_{2k} = \ell_{2k} + 6\ell_{2k-2} + \ell_{2k-4},
\qquad
b_{2k+1} = 4\ell_{2k} + 4\ell_{2k-2},
\]
we compute, for example,
\begin{align*}
    b_2 &= \ell_2 + 6\ell_0 = 32 + 6(1) = 38, \\
    b_3 &= 4\ell_2 + 4\ell_0 = 4(32) + 4(1) = 132, \\
    b_4 &= \ell_4 + 6\ell_2 + \ell_0 = 328 + 6(32) + 1 = 521.
\end{align*}
Proceeding similarly for the remaining degrees, we obtain the complete Betti sequence:
\[
(b_0,b_1,\dots,b_{20}) =
(1,4,38,132,521,1440,3184,6048,9242,11976,13228,\dots,1).
\]
\end{example}

\subsection{Comparison with compact hyper-Kähler manifolds}

It is known that in any $4n$-dimensional compact hyper-K\"ahler manifold the following conditions hold:
\begin{enumerate}
    \item[(i)] the odd-degree Betti numbers $b_{2k+1}$ are divisible by $4$ (\cite{Wak}),
    \item[(ii)] its Betti numbers satisfy the identity (\cite{Sal1, Sal}):
\begin{equation}\label{eq: salamon}
 n\, b_{2n}=2 \sum_{j=1}^{2n} (-1)^j (3j^2-n)b_{2n-j}.
\end{equation}
\end{enumerate}
We note that according to \cite[Main Theorem]{Hasegawa}, the hypercomplex solvmanifolds $M_p^h$ do not admit any hyper-K\"ahler metric for any $p\in \Delta_n$, since they are quotients of a completely solvable Lie group. However, it follows from Theorem \ref{thm: cohomology_hyp} and \eqref{eq: betti_impar} that their Betti numbers satisfy condition (i) whenever $p$ has full rank or quasi full rank. Furthermore, after lengthy computations, we see that \eqref{eq: salamon} is valid for $M_p^h$ when $p$ satisfies the full rank condition. 

We will show now that, more generally, Salamon's identity \eqref{eq: salamon} holds for $M_p^h$ for any $p \in \Delta_n$. This will be a consequence of the following two results, the second inspired by the work of Salamon in \cite{Sal1, Sal}. 

\begin{lemma}
    For any $p \in \Delta_n$, the Poincaré polynomial of any hypercomplex solvmanifold $M_p^h$ is divisible by $(1+t)^4$.
\end{lemma}

\begin{proof}
    As a consequence of \eqref{eq: kunneth_hypercomplex}, we easily find that
    \[
    P_{\g_p^h}(t) = (1+t)^3 P_{\bar{\g}_p}.
    \]
    Since $\bar{\g}_p$ is unimodular and odd-dimensional, it follows that $P_{\bar{\g}_p} (-1) = 0$ and then
    \[
    P_{\g_p^h}(t) = (1+t)^4 Q(t) 
    \]
    with $Q \in \Z[t]$.
\end{proof}

\begin{proposition}
    Let $M^{4n}$ be a compact hypercomplex manifold. If the Poincaré polynomial $P_M$ of $M$ is divisible by $(1+t)^3$, then the Betti numbers of M satisfy Salamon's identity \eqref{eq: salamon}. 
\end{proposition}

\begin{proof}
    Let $P(t) = \sum_{k=0}^{4n} b_k t^k$ be the Poincaré polynomial of $M$. We know that $P(t)$ is a self-reciprocal polynomial since $M$ satisfies Poincaré duality ($b_k = b_{4n-k}$, with $0 \leq k \leq 2n$); let us assume that it is divisible by $(1+t)^3$. This condition implies $P(-1) = P'(-1) = P''(-1) = 0$.
    
    Evaluating $P(-1) = 0$ gives:
    \[
    \sum_{k=0}^{4n} (-1)^k b_k = 0.
    \]
    We can center this sum around the middle dimension by changing the index to $k = 2n - i$, with $-2n \leq i \leq 2n$. Using the symmetry $b_{2n-i} = b_{2n+i}$, we obtain:
    \begin{equation}\label{eq: p(-1)}
    b_{2n} + 2 \sum_{i=1}^{2n} (-1)^i b_{2n-i} = 0. 
    \end{equation}

    Now we consider the derivatives. The condition $P'(-1)=0$ implies $\sum k(-1)^k b_k = 0$, while $P''(-1)=0$ implies $\sum k(k-1)(-1)^k b_k = 0$. Adding these two equations yields:
    \[
    \sum_{k=0}^{4n} k^2 (-1)^k b_k = 0.
    \]
    
    Applying the same index shift $k = 2n - i$, the equation becomes:
    \begin{align*}
        0 & = \sum_{i=-2n}^{2n} (2n-i)^2 (-1)^i b_{2n-i}\\ & = \sum_{i=-2n}^{2n} (4n^2 - 4ni + i^2) (-1)^i b_{2n-i} \\
        & = 4n^2 \sum_{i=-2n}^{2n} (-1)^i b_{2n-i} - 4n \sum_{i=-2n}^{2n} i(-1)^i i b_{2n-i} + \sum_{i=-2n}^{2n} i^2(-1)^i b_{2n-i}.
    \end{align*}
  
    The first sum vanishes due to \eqref{eq: p(-1)}. The second sum cancels out completely due to the symmetry of the Betti numbers (the term for $i$ precisely cancels the term for $-i$). We are thus left only with the third sum. Extracting the term $i=0$, we get:
    \begin{equation} \label{eq: third_sum}
        2 \sum_{i=1}^{2n} i^2 (-1)^i b_{2n-i} = 0.
    \end{equation}

    The Salamon identity in dimension $4n$ states that $n b_{2n} = 2 \sum_{i=1}^{2n} (-1)^i (3i^2 - n) b_{2n-i}$. Rearranging all terms to one side, this condition is equivalent to:
    \[
    n \left( b_{2n} + 2 \sum_{i=1}^{2n} (-1)^i b_{2n-i} \right) - 3 \left( 2 \sum_{i=1}^{2n} i^2 (-1)^i b_{2n-i} \right) = 0,
    \]
    and this clearly holds due to \eqref{eq: p(-1)} and \eqref{eq: third_sum}. 
\end{proof}

\begin{remark}
    There are examples of hypercomplex solvmanifolds whose cohomology does not behave like the cohomology of a compact hyper-Kähler manifold. Indeed, consider a nilmanifold $M$ obtained as a quotient of the 8-dimensional hypercomplex nilpotent Lie group $N_3$ from \cite{DF1}. Then its Betti numbers are $(1,5,15,25,28,25,15,5,1)$, which do not satisfy either (i) or (ii) above. Note also that the Poincaré polynomial of $M$ factors as
    \[ P_M(t)=(t+1)^2(t^2+1)(t^4+3t^3+7t^2+3t+1),\]
    so that $-1$ is only a double root of $P_M$. The nilpotent Lie group $N_3$ is the direct product of the 7-dimensional quaternionic Heisenberg group with the line $\R$.
\end{remark}

\section{Dolbeault cohomology of the complex solvmanifolds} \label{sec: dolbeault}

\subsection{Nakamura manifolds}
In this section, we relate our complex solvmanifolds $\Gamma^c\backslash G_p^c$ with the Nakamura manifolds introduced recently in \cite{CT}, inspired by the foundational work of I. Nakamura in \cite{Na}\footnote{The notion of Nakamura manifold was further generalized by A.\ Cattaneo very recently in \cite{Ca}.}.

Recall that a Nakamura manifold is constructed as a quotient of a completely solvable Lie group $G_M = \C \ltimes_\rho \C^n$, arising from a matrix $M\in \operatorname{SL}(n,\Z)$ that is diagonalizable over $\R$ with all its eigenvalues positive. Thus, there exists $P\in \operatorname{GL}(n,\R)$ such that 
\[ PMP^{-1}=\operatorname{diag}(\e^{\lambda_1},\ldots,\e^{\lambda_n})\]
for some real numbers $\lambda_i$, not all of them $0$, such that
\[ \sum_{i=1}^n \lambda_i =0.\]
We then consider the representation $\rho\colon \C \to \operatorname{GL}(n,\C)$ given by
\[ \rho(w)=\operatorname{diag}(\e^{\lambda_1 \operatorname{Re}w}, \ldots,\e^{\lambda_n \operatorname{Re}w} ),\]
which allows us to define the semidirect product $G_M=\C\ltimes \C^n$. The Lie group $G_M$ admits lattices; indeed, for any $\tau\in\R-\{0\}$, set 
\[ \Gamma'_\tau:=\Z\oplus i\tau \Z\subset \C, \qquad \Gamma''_P:=P\Z^n\oplus iP \Z^n\subset \C^n.\]
Since $\rho(\Gamma'_\tau)$ preserves $\Gamma''_P$, we can define $\Gamma_{P,\tau}:=\Gamma'_\tau\ltimes_\rho \Gamma''_P$, which can be seen to be a lattice of $G_M$. The solvmanifold $\Gamma_{P,\tau}\backslash G_M$ is called a \textit{Nakamura manifold}.

We point out that there is a group isomorphism $\Gamma_{P,\tau}\cong \Gamma_{P,\tau'}$ for any $\tau,\tau'\in \R-\{0\}$, so that $\Gamma_{P,\tau}\backslash G_M$ is diffeomorphic to $\Gamma_{P,\tau'}\backslash G_M$, according to Theorem \ref{thm:solv-isom}. However, we will see below that the complex geometry of the Nakamura manifolds does depend on the nonzero real number $\tau$.

\medskip

At the Lie algebra level, it was shown in \cite[Section 4.5]{CT} that the Lie algebra of $G_M$ has a basis $\{e_0,f_0,e_1,\ldots, e_n,f_1,\ldots, f_n\}$ with Lie bracket given by
\begin{equation}\label{eq: brackets} [e_0,e_i]=\lambda_i e_i, \quad [e_i,f_i]=\lambda_i f_i, \quad i=1,\ldots,n.
\end{equation}
Note that the real numbers $\lambda_i$ are the logarithms of the roots of the characteristic polynomial of $M$. 

This establishes that our Lie group $G_p^c$ arising from a polynomial $p\in \Delta'_n$ is isomorphic to a Lie group $G_{C_p}$ as above, where $C_p\in \operatorname{SL}(n,\Z)$ denotes the companion matrix of the polynomial $p$. Indeed, the Lie brackets in \eqref{eq: brackets} are exactly the Lie brackets for $\g_p^c= \R e_0 \ltimes_{A_p^c} \R^{2n+1}$, where $A_p^c = 0_1 \oplus A_p^{\oplus 2}$, under the identification $\e^{\lambda_i}=r_i$, with $r_1,\ldots,r_n$ the roots of $p$. Note that this isomorphism of Lie groups is actually a biholomorphism.

Recall from Section \ref{sec: construction} the matrix $Q_p\in \operatorname{GL}(n,\R)$ that satisfies $Q_p C_p Q_p^{-1}=\operatorname{diag}(r_1,\ldots,r_n)$. Then, the lattice $\Gamma_{Q_p,\tau}$ of $G_{C_p}$ as above corresponds to the following lattice in $G^c_p=\R\ltimes \R^{2n+1}$:
\[ \Gamma^c_{p,\tau}:=\{(p,(\tau q, Q_p v,Q_p w))\in G^c_p \colon p, q\in \Z,\,  v,w\in \Z^n\}. \]
Note that for $\tau=1$ we have $\Gamma^c_{p,1}=\Gamma^c_p$ from Section \ref{sec: construction}, constructed using Bock's result (Proposition \ref{prop: Bock}). 
Thus, the almost abelian solvmanifolds $\Gamma_{p,\tau}\backslash G^c_p$ are Nakamura manifolds in the sense of \cite{CT}.

\medskip

Regarding the Dolbeault cohomology of the Nakamura manifolds, it was shown in \cite[Section 4.4]{CT} that there exist certain values of $\tau\in \R-\{0\}$ for which the Dolbeault cohomology of the Nakamura manifold $\Gamma_{P,\tau}\backslash G_M$ can be computed using invariant forms, and moreover, an expression for the associated Hodge numbers $h^{r,s}$ was given. We will state these results more explicitly for our manifolds $\Gamma_{p,\tau}\backslash G^c_p$, which are Nakamura manifolds: given $p\in \Delta'_n$, consider subsets $I,J\subseteq \{1,\ldots,n\}$ with $|I|=r$ and $|J|=s$. 
Next, define
\[ c_{IJ}=\sum_{i \in I} s_i + \sum_{j \in J} s_j,\]
where, as usual, $s_i=\log r_i$ and $r_1,\ldots,r_n$ are the roots of $p$, and let us set 
\begin{equation}\label{eq: Nrs} 
N(r,s)=|\{(I,J) \colon |I|=r,\, |J|=s \text{ and } c_{IJ}=0\}|. 
\end{equation}

Consider the following condition on $\tau$:
\begin{equation}\label{eq: cIJ}
    \text{for any }I, J \text{ it holds that: } \tau c_{IJ}\in 2\pi \Z \Longleftrightarrow c_{IJ}=0. 
\end{equation}
It was shown in \cite{CT} that nonzero real numbers $\tau$ satisfying condition \eqref{eq: cIJ} exist. In the next theorem we summarize several results proved in \cite{CT}, adapted to our case:

\begin{theorem}\label{thm: summary}
    Let $p \in \Delta'_n$ be a polynomial with associated logarithmic parameters $s_i = \log r_i$ for $i = 1, \dots, n$, and assume that $\tau\in \R-\{0\}$ satisfies condition \eqref{eq: cIJ}. Then the Nakamura manifolds $M^c_{p,\tau}=\Gamma_{p,\tau}\backslash G^c_p$ have the following properties:
    \begin{enumerate}
    \item[\rm{(i)}] complex conjugation induces an isomorphism 
    \[ \overline{H^{r,s}_{\overline{\partial}}(\Gamma_{p,\tau}\backslash G^c_p)}\cong H^{s,r}_{\overline{\partial}}(\Gamma_{p,\tau}\backslash G^c_p),\]
    \item[\rm{(ii)}] the Frölicher spectral sequence degenerates at the $E_1$ page,
    \item[\rm{(iii)}] the $\partial\overline{\partial}$-Lemma is satisfied,
    \item[\rm{(iv)}] its Hodge numbers are given by
    \begin{equation}\label{eq: hodge_formula} h^{r,s} = N(r,s) + N(r-1, s) + N(r, s-1) + N(r-1, s-1),
    \end{equation}
    where $N(a, b) = 0$ if $a < 0$ or $b < 0$,
    \item[\rm{(v)}] its Betti numbers satisfy $b_k=\sum_{r+s=k} h^{r,s}$.
    \end{enumerate} 
\end{theorem}

In what follows, we will compute the Hodge numbers of the Nakamura solvmanifolds $M_{p,\tau}^c = \Gamma^c_{p,\tau} \backslash G_p^c$, where $\tau$ satisfies the condition \eqref{eq: cIJ}, under the usual hypothesis of $p$ being full rank or quasi full rank. Assuming that the polynomial $p$ satisfies either of these hypotheses, the combinatorial terms $N(r,s)$ admit an explicit description. 

\subsection{The full rank case}

\begin{lemma} \label{lemma: N_pq}
    Let $p \in \Delta'_n$ be a polynomial with associated logarithmic parameters $s_i = \log r_i$ for $i = 1, \dots, n$. If $p$ satisfies the full rank condition, then:
    \begin{equation*}
        N(r, s) = \begin{cases}
            1 & \text{if } (r, s) = (0, 0) \text{ or } (n, n), \\[0.3em]
            \binom{n}{r} & \text{if } r + s = n, \\
            0 & \text{otherwise}.
        \end{cases}
    \end{equation*}
\end{lemma}

\begin{proof}
    For each parameter $s_i$, define $c_i \in \{0, 1 ,2\}$ to be its multiplicity in the sum associated with the pair $(I,J)$, where $|I|=r$ and $|J|=s$. More precisely, $c_i = 1$ if $i$ belongs to exactly one of $I$ and $J$, $c_i = 2$ if $i$ belongs to both, and $c_i = 0$ otherwise. Then, the condition $c_{IJ}=0$ becomes 
    \[ 
    \sum_{i=1}^{n} c_i s_i = 0. 
    \] 
    Since $p$ satisfies the full rank condition, this equality holds if and only if all the coefficients $c_i$ are equal to a constant $c \in \{0, 1, 2\}$.
    
    If $c=0$, then $c_i=0$ for every $i$,  which implies that $I=J=\emptyset$. Hence, $(r,s)=(0,0)$ and 
    \[ 
    N(0,0)=1. 
    \] 
    
    If $c=1$, then each index belongs to exactly one of the sets $I$ and $J$. Thus, $I$ and $J$ form a partition of $\{1,\dots,n\}$ and $r+s=n$. Moreover, once $I$ is chosen, the set $J$ is uniquely determined as its complement. Therefore, 
    \[ 
    N(r,s)=\binom{n}{r} \qquad \text{whenever } r+s=n. 
    \] 
    
    Finally, if $c=2$, then $c_i=2$ for all $i$, meaning that every index belongs to both sets. Hence, $I=J=\{1,\dots,n\}$, yielding $(r,s)=(n,n)$ and 
    \[ 
    N(n,n)=1. 
    \]
 
    In all other combinations for $r$ and $s$, no such constant $c$ exists, yielding $N(r, s) = 0$.
\end{proof}

\begin{theorem} \label{thm: hodge_complex}
    Let $M_{p,\tau}^c$ be the complex solvmanifold of complex dimension $n+1$ associated with a polynomial $p \in \Delta'_n$ satisfying the full rank condition and with $\tau$ satisfying condition \eqref{eq: cIJ}. Let $h^{r,s} = \dim H^{r,s}_{\bar{\partial}}(M_p^c)$ denote the Hodge numbers of $M_p^c$. Then:
    \begin{itemize}
        \item \textbf{General Case ($n \geq 3$):}
        The nonzero Hodge numbers are:
        \begin{align*}
            &h^{0,0} = h^{1,0} = h^{0,1} = h^{1,1} = 1, \\
            &h^{r, n-r} = \binom{n}{r}, \quad 0\leq r \leq n, \\
            &h^{r, n+1-r} = \binom{n+1}{r}, \quad 0\leq r \leq n+1,\\
            &h^{r, n+2-r} = \binom{n}{r-1}, \quad 1\leq r \leq n+1,\\
            &h^{n,n} = h^{n+1,n} = h^{n,n+1} = h^{n+1,n+1} = 1.
        \end{align*}
        
        \item \textbf{Particular Case ($n = 2$):}
        The complex dimension is $3$ and the Hodge numbers are: 
        \begin{align*}
            &h^{0,0} = 1, \\ &h^{1,0} = h^{0,1} = 1, \\
            &h^{2,0} = h^{0,2} = 1, \quad h^{1,1} = 3, \\
            &h^{3,0} = h^{0,3} = 1, \quad h^{2,1} = h^{1,2} = 3, \\
            &h^{3,1} = h^{1,3} = 1, \quad h^{2,2} = 3, \\
            &h^{3,2} = h^{2,3} = 1, \\
            &h^{3,3} = 1.
        \end{align*}
    \end{itemize}
\end{theorem}

\begin{proof}
We compute the Hodge numbers of $M_{p,\tau}^c$ using the general formula \eqref{eq: hodge_formula} from Theorem \ref{thm: summary} (iv) and the values of the combinatorial coefficients $N(r,s)$ described in Lemma \ref{lemma: N_pq}.

Assume first that $n \geq 3$. In this case, $N(r,s)$ is nonzero only at the points $(0,0)$ and $(n,n)$, and for indices satisfying 
\[ 
r+s=n. 
\] 
Since the four terms appearing in \eqref{eq: hodge_formula} involve the values \[ 
N(r,s), \qquad N(r-1,s), \qquad N(r,s-1), \qquad N(r-1,s-1), 
\] 
their arguments lie in a block $2\times 2$ whose coordinate sums range from $r+s-2$ to $r+s$. If such a block contains the origin $(0,0)$, then its maximal coordinate sum is $2$, which cannot meet the relation $i+j=n$ when $n\geq 3$. Conversely, if the block contains $(n,n)$, then its minimal coordinate sum is $2n-2>n$.

It follows that, for $n\geq 3$, the contributions arising from $(0,0)$, $(n,n)$, and the line $i+j=n$ are pairwise disjoint. Consequently, the four terms in \eqref{eq: hodge_formula} can be analyzed independently. 

The contribution of $(0,0)$ yields 
\[ 
h^{0,0} = h^{1,0} = h^{0,1} = h^{1,1} = 1, 
\] 
while the contribution from $(n,n)$ gives 
\[ 
h^{n,n} = h^{n+1,n} = h^{n,n+1} = h^{n+1,n+1} = 1. 
\] 

If $r+s=n$, only the term $N(r,s)$ is nonzero, and therefore 
\[ 
h^{r,n-r} = \binom{n}{r}, \qquad 0\leq r\leq n. 
\] 

If $r+s=n+1$, the only nonzero contributions are $N(r-1,s)$ and $N(r,s-1)$. Writing $s=n+1-r$, Lemma \ref{lemma: N_pq} gives 
\[ 
h^{r,n+1-r} = \binom{n}{r-1} + \binom{n}{r}. 
\] 
Applying Pascal's identity, we obtain 
\[ 
h^{r,n+1-r} = \binom{n+1}{r}, \qquad 0\leq r\leq n+1. 
\] 

Finally, if $r+s=n+2$, only the term $N(r-1,s-1)$ contributes, yielding 
\[ 
h^{r,n+2-r} = \binom{n}{r-1}, \qquad 1\leq r\leq n+1. 
\]

For $n=2$, the diagonal $r+s=2$ is no longer disjoint from the contributions arising from $(0,0)$ and $(2,2)$. In particular, 
\[ 
N(2,0)=1, \qquad N(1,1)=2, \qquad N(0,2)=1. 
\] 
A direct evaluation of \eqref{eq: hodge_formula} then yields the Hodge numbers listed in the statement. For example, 
\[ 
h^{1,1} = N(1,1) + N(0,1) + N(1,0) + N(0,0) = 2+0+0+1 = 3, 
\] 
and the remaining values are obtained similarly.
\end{proof}

\begin{remark}
When $n=2$, the solvmanifolds $M^c_{p,\tau}$ correspond to the completely solvable Nakamura manifolds appearing in case (C) in \cite[Section 5.1]{Kas}, and their Hodge numbers coincide with the values found in Theorem \ref{thm: cohomology_complex}. The Hodge diamond for these manifolds is given by:   
\[
\begin{array}{ccccccc}
&&& 1 &&& \\
&& 1 && 1 && \\
& 1 && 3 && 1 & \\
1 && 3 && 3 && 1 \\
& 1 && 3 && 1 & \\
&& 1 && 1 && \\
&&& 1 &&&
\end{array}
\]
\end{remark}

\begin{example}
The Hodge diamond for the case $n=4$ (that is, complex dimension equal to 5), is shown in the figure below.
\[
\begin{array}{ccccccccccc}
&&&&& 1 &&&&& \\
&&&& 1 && 1 &&&& \\
&&& 0 && 1 && 0 &&& \\
&& 0 && 0 && 0 && 0 && \\
& 1 && 4 && 6 && 4 && 1 & \\
1 && 5 && 10 && 10 && 5 && 1 \\
& 1 && 4 && 6 && 4 && 1 & \\
&& 0 && 0 && 0 && 0 && \\
&&& 0 && 1 && 0 &&& \\
&&&& 1 && 1 &&&& \\
&&&&& 1 &&&&& \\
\end{array}
\]
\end{example}

In general, for any $n \geq 3$ (corresponding to  complex dimension $\geq 4$), the Hodge diamond exhibits a highly symmetric structure related to Pascal's triangle. Specifically, the central horizontal row (defined by the Hodge numbers $h^{r,s}$ with $r+s=n+1$) coincides exactly with the $(n+1)$-th row of Pascal's triangle. This combinatorial pattern extends to the adjacent rows immediately above and below, where $r+s=n$ and $r+s=n+2$, respectively. Outside of these three rows, several lines of Hodge numbers vanish entirely, with the exception of two $2 \times 2$ blocks at the extremities: the top corner where $h^{0,0} = h^{1,0} = h^{0,1} = h^{1,1} = 1$, and the bottom corner where $h^{n,n} = h^{n+1,n} = h^{n,n+1} = h^{n+1,n+1} = 1$.

\begin{remark}
As stated in Theorem \ref{thm: summary}(v), the identity
\[
    b_k = \sum_{r+s=k} h^{r,s}
\]
holds for any Nakamura manifold as long as we choose the nonzero real number $\tau$ satisfying condition \eqref{eq: cIJ}. In the full rank setting we can also verify this by comparing the Betti numbers of Theorem \ref{thm: cohomology_complex} with the Hodge numbers $h^{r,s}$ in Theorem \ref{thm: hodge_complex}. Indeed, for $n \geq 3$, the nontrivial sums of Hodge numbers are
\begin{align*}
\sum_{r+s=n} h^{r,s}
    &= \sum_{r=0}^n \binom{n}{r}
    = 2^n, \\
\sum_{r+s=n+1} h^{r,s}
    &= \sum_{r=0}^{n+1} \binom{n+1}{r}
    = 2^{n+1}, \\
\sum_{r+s=n+2} h^{r,s}
    &= \sum_{r=1}^{n+1} \binom{n}{r-1}
    = 2^n,
\end{align*}
which coincide with $b_n$, $b_{n+1}$, and $b_{n+2}$, respectively. Likewise,
\[
b_0=b_{2n+2}=1, \qquad
b_1=b_{2n+1}=2, \qquad
b_2=b_{2n}=1,
\]
and all remaining Betti and Hodge numbers vanish.
\end{remark}

\subsection{The quasi full rank case}
Let $p \in \Delta'_n$ be a self-reciprocal polynomial of even degree $n = 2m$. The generating function approach used for the Betti numbers can be naturally extended to the computation of the Hodge numbers $h^{p,q}$ of the complex solvmanifold $M_{p,\tau}^c$. Recall from \eqref{eq: Nrs} that $N(r,s)$ counts the number of pairs of subsets $I,J \subseteq \{1,\dots,n\}$ with $|I|=r$ and $|J|=s$ such that
\[
\sum_{i\in I} s_i + \sum_{j\in J} s_j =0.
\]
We encode these quantities into the bivariate generating polynomial
\[
Q(x,y)=\sum_{r,s} N(r,s)x^r y^s,
\]
where the exponents of $x$ and $y$ record the cardinalities of the subsets $I$ and $J$, respectively. This generating function will provide an effective way to compute the Hodge polynomial
\[
H(x,y)=\sum_{r,s} h^{r,s}x^r y^s
\]
in the quasi full rank setting.

\begin{theorem} \label{thm: hodge_generating_function}
    Let $M_{p,\tau}^c$ be the complex solvmanifold of complex dimension $n+1$ associated with a polynomial $p \in \Delta'_n$ ($n=2m$) satisfying the quasi full rank condition and with $\tau$ satisfying condition \eqref{eq: cIJ}. Then the Hodge polynomial of $M_{p,\tau}^c$ is given by:
    \begin{equation*} \label{eq: bivariate_hodge}
        H(x,y) = (1+x)(1+y) \left( 1 + (x+y)^2 + x^2 y^2 \right)^m.
    \end{equation*}
\end{theorem}

\begin{proof}
Since $p$ is self-reciprocal, the set of logarithmic parameters is given by
\[
S=\{s_1,\dots,s_m,-s_1,\dots,-s_m\}.
\]
For each $i=1,\dots,m$, let
\[
B_i=\{i,m+i\}
\]
be the pair of indices corresponding to the parameters $s_i$ and $-s_i$. Given subsets
\[
I,J\subset\{1,\dots,n\},
\]
define
\[
I_i=I\cap B_i,
\qquad
J_i=J\cap B_i.
\]
Then, condition $c_{IJ}=0$ holds if and only if it is satisfied independently on each block $B_i$.

We encode the contribution of a pair $(I,J)$ by the monomial $x^{|I|}y^{|J|}$.
Since the blocks are independent, the problem reduces to determining the contribution of a single block $B_i=\{i,m+i\}$, corresponding to the pair of parameters $\{s_i,-s_i\}$. 

There are only three possible configurations for $(I_i,J_i)$ for which the sum of the parameters indexed by $I_i$ and $J_i$ vanishes:

\begin{itemize} 
    \item If $I_i=J_i=\emptyset$, the contribution is $1$. 
    \item If exactly one index corresponding to $s_i$ and one corresponding to $-s_i$ are selected across $I_i$ and $J_i$, then the vanishing condition forces both to appear exactly once. Since each may belong independently to either $I_i$ or $J_i$, the four admissible possibilities contribute 
    \[ 
    x^2+2xy+y^2 = (x+y)^2. 
    \] 
    \item If both indices in $B_i$ belong simultaneously to $I_i$ and to $J_i$, then $|I_i|=|J_i|=2$, yielding the contribution
    \[ 
    x^2y^2. 
    \] 
    \end{itemize}

Therefore, the generating polynomial associated with a single block is 
\[ 
Q_{\mathrm{block}}(x,y) = 1 + (x + y)^2 + x^2 y^2. 
\]

Since the $m$ blocks are independent, the global generating polynomial factorizes as 
\[ 
Q(x,y) = \bigl(Q_{\mathrm{block}}(x,y)\bigr)^m. 
\] 

Finally, by equation \eqref{eq: hodge_formula} from Theorem \ref{thm: summary}(iv),
\[
h^{r,s} = N(r,s) + N(r-1,s) + N(r,s-1) + N(r-1,s-1).
\]

At the level of generating polynomials, multiplication by $x$ shifts the first index by one, while multiplication by $y$ shifts the second index by one. Thus, the coefficient of $x^r y^s$ in $xQ(x,y)$ equals $N(r-1,s)$, while the corresponding coefficients in $yQ(x,y)$ and $xyQ(x,y)$ equal $N(r,s-1)$ and $N(r-1,s-1)$, respectively. Hence 
\[ 
H(x,y) = Q(x,y) + xQ(x,y) + yQ(x,y) + xyQ(x,y) = (1+x)(1+y)Q(x,y). 
\] 

Substituting the expression for $Q(x,y)$ yields 
\[ 
H(x,y) = (1+x)(1+y) \left( 1+(x+y)^2+x^2y^2 \right)^m, 
\] 
as claimed.
\end{proof}

\begin{remark}
    The identity
    \[
        b_k=\sum_{r+s=k} h^{r,s}
    \]
    can also be easily verified in the quasi full rank setting. In this case, we can use the generating polynomials: setting $x=y=t$ in the Hodge polynomial yields
    \[
        H(t,t) = \sum_{r,s} h^{r,s} t^{r+s} = \sum_k \left(\sum_{r+s=k} h^{r,s}\right)t^k.
    \]
    Using Theorem \ref{thm: hodge_generating_function}, we obtain
    \[
        H(t,t) = (1+t)^2(1+4t^2+t^4)^m,
    \]
    which coincides with the Poincaré polynomial $P_{M_{p, \tau}^c}(t)$ obtained in Corollary \ref{cor: poincare_polynomials}. Consequently, $b_k=\sum_{r+s=k} h^{r,s}$ for all $k$.
\end{remark}

\begin{remark}
    Let $p \in \Delta'_n$ be a quasi full rank polynomial of degree $n=2m$. Expanding the generating polynomial
    \[
        Q(x,y) = \bigl(1 + (x+y)^2 + x^2 y^2\bigr)^m,
    \]
    one observes that the combinatorial coefficients $N(r,s)$ vanish whenever $r$ and $s$ have opposite parity; namely,
    \[
        N(r,s)=0 \qquad \text{if } r \not\equiv s \pmod 2.
    \]
    Consequently, the general expression for the Hodge numbers $h^{r,s}$ in \eqref{eq: hodge_formula} simplifies to
    \begin{equation} \label{eq: hodge_explicit_cases}
        h^{r,s} =
        \begin{cases}
            N(r,s) + N(r-1,s-1), & \text{if } r \equiv s \pmod 2, \\[4pt]
            N(r-1,s) + N(r,s-1), & \text{if } r \not\equiv s \pmod 2.
        \end{cases}
    \end{equation}
\end{remark}

\begin{corollary}\label{cor: hodge_porq_zero}
    The Hodge numbers of type $(r,0)$ and $(0,r)$ are given by:
    \begin{equation*}
        h^{r,0} = h^{0,r} = \binom{m}{\lfloor r/2 \rfloor}, \quad \text{for } 0 \le r \le 2m+1.
    \end{equation*}
\end{corollary}

\begin{proof}
Setting $y=0$ gives $Q(x,0) = (1 + x^2)^m$, hence $N(r,0)=\binom{m}{r/2}$ if $r$ is even and $0$ otherwise. From \eqref{eq: hodge_explicit_cases} with $q=0$, we obtain $h^{r,0}=N(r,0)$ for $r$ even and $h^{r,0}=N(r-1,0)$ for $r$ odd, which yields $h^{r,0}=\binom{m}{\lfloor r/2 \rfloor}$. The identity for $h^{0,r}$ follows by symmetry.
\end{proof}

\begin{example}
    Let $p \in \Delta'_4$ be a quasi full rank polynomial. Then, expanding the Hodge polynomial $H(x,y)$ from Theorem \ref{thm: qfr_generating_function} with $m=2$ we obtain: 
    \begin{align*}
        H(x,y) &= x^5 y^5 + x^5 y^4 + 2 x^5 y^3 + 2 x^5 y^2 + x^5 y + x^5 \\
        & \phantom{=} + x^4 y^5 + 5 x^4 y^4 + 6 x^4 y^3 + 6 x^4 y^2 + 5 x^4 y + x^4 \\
        & \phantom{=} + 2 x^3 y^5 + 6 x^3 y^4 + 12 x^3 y^3 + 12 x^3 y^2 + 6 x^3 y + 2 x^3 \\
        & \phantom{=} + 2 x^2 y^5 + 6 x^2 y^4 + 12 x^2 y^3 + 12 x^2 y^2 + 6 x^2 y + 2 x^2 \\
        & \phantom{=} + x y^5 + 5 x y^4 + 6 x y^3 + 6 x y^2 + 5 x y + x \\
        & \phantom{=} + y^5 + y^4 + 2 y^3 + 2 y^2 + y + 1.
    \end{align*}
    Thus, the Hodge diamond for the manifold $M_{p, \tau}^c$ is given by:
    \[
\begin{array}{ccccccccccc}
&&&&& 1 &&&&& \\
&&&& 1 && 1 &&&& \\
&&& 2 && 5 && 2 &&& \\
&& 2 && 6 && 6 && 2 && \\
& 1 && 6 && 12 && 6 && 1 & \\
1 && 5 && 12 && 12 && 5 && 1 \\
& 1 && 6 && 12 && 6 && 1 & \\
&& 2 && 6 && 6 && 2 && \\
&&& 2 && 5 && 2 &&& \\
&&&& 1 && 1 &&&& \\
&&&&& 1 &&&&&
\end{array}
\]
    
\end{example}

\section{Other examples}

Throughout this section, we freely identify subsets of indices with the corresponding sub-multisets of logarithmic parameters, as in the previous sections. Accordingly, $\ell_k$ denotes the number of sub-multisets of cardinality $k$ whose elements sum to zero.

\subsection{Example with a polynomial that does not satisfy the full rank condition}

We construct a polynomial $p \in \Delta'_5$ that does not satisfy the full rank condition. In degree $5$, the only possible non-trivial multiplicative relations (besides the total product) have lengths $2$ and $3$. The following example realizes exactly one relation of each type.

Let $p_1(x)=x^2-3x+1 \in \Delta'_2$, whose roots form a set $A=\{r_1,r_2\}$ satisfying $r_1r_2=1$; and $p_2(x)=x^3-5x^2+6x-1 \in \Delta'_3$, with roots $B=\{r_3,r_4,r_5\}$ satisfying $r_3r_4r_5=1$.
Define
\[
p(x) = p_1(x)p_2(x) = x^5 - 8x^4 + 22x^3 - 24x^2 + 9x - 1.
\]

Since $p_1$ is irreducible over $\Q$ and does not divide $p_2$, the polynomials $p_1$ and $p_2$ are coprime. Hence,
\[
A\cap B=\emptyset.
\]

We now determine the integers $\ell_k$ associated with the solvmanifold $M_p$:
\begin{itemize}
    \item \textbf{Lengths $k=1$ and $k=4$.}
    Since $p \in \Delta'_5$, $p(1) = 1\neq 0$ and $\ell_1 = 0$. By taking complements, it follows that $\ell_4 = 0$.

    \item \textbf{Lengths $k=2$ and $k=3$.}
    The sets $A$ and $B$ provide one relation of lengths $2$ and $3$, respectively. We show that mixed relations do not occur.

    Suppose first that
    \[
    xy=1,
    \qquad x\in A,\ y\in B.
    \]
    Then $y=x^{-1}$. Since $p_1$ is self-reciprocal, we have $x^{-1}\in A$, and hence $y \in A \cap B$, contradicting $A\cap B = \emptyset$. Therefore, there are no mixed relations of length $2$.

    For relations of length $3$, there are two possibilities:
    \begin{enumerate}
        \item [(i)] If two roots belong to $A$ and one root belongs to $B$, then $r_1 r_ 2y = 1$. But $r_1 r_2 = 1$, so $y = 1$, contradicting the fact that $1\notin B$.

        \item [(ii)] If one root belongs to $A$ and two roots belong to $B$, then
        \[
        x(yz)=1,
        \qquad x\in A,\ y,z\in B.
        \]
        Let $w\in B$ be the remaining root of $p_2$. Since $yzw = 1$, we obtain $yz = w^{-1}$. Substituting into the previous relation gives $xw^{-1} = 1$, and therefore $x=w$, again contradicting $A\cap B=\emptyset$.
    \end{enumerate}

    Hence, there are no mixed relations of length $3$. Consequently, $\ell_2=\ell_3=1$.
\end{itemize}

The only nonzero values are $\ell_0=\ell_2=\ell_3=\ell_5=1$. Using Lemma \ref{lemma: betti_combinatorial}, $b_k=\ell_k+\ell_{k-1}$, the Betti numbers of the $6$-dimensional basic solvmanifold $M_p$ are
\[
b(M_p)=(1,1,1,2,1,1,1).
\]

We now compute the Betti numbers of the associated complex and hypercomplex solvmanifolds using generating polynomials.

For the 12-dimensional complex solvmanifold $M_p^c$, each root has multiplicity 2. The relations coming from the sets $A$ and $B$ are encoded by
\begin{align*}
L_A(t)
    &=\sum_{i=0}^{2}\binom{2}{i}^2 t^{2i}
      =1+4t^2+t^4,\\
L_B(t)
    &=\sum_{j=0}^{2}\binom{2}{j}^3 t^{3j}
      =1+8t^3+t^6.
\end{align*}

The coefficients $\ell_k $ of the product $L_A(t)L_B(t)$ determine the Betti numbers of the associated 11-dimensional almost abelian Lie algebra $\tilde{\g}_p$ through Lemma \ref{lemma: betti_combinatorial}. Then, using the identity \eqref{eq: kunneth_complex} derived from the decomposition $\g_p^c = \tilde{\g}_p \times \R f$, we obtain
    \begin{equation*}
        b(M_p^c) = (1, 2, 5, 16, 21, 42, 66, 42, 21, 16, 5, 2, 1).
    \end{equation*}

For the 24-dimensional basic hypercomplex solvmanifold $M_p^{h}$, each root appears with multiplicity $4$. The corresponding generating functions are
\begin{align*}
L_A(t)
    &=\sum_{i=0}^{4}\binom{4}{i}^2 t^{2i}
      =1+16t^2+36t^4+16t^6+t^8,\\
L_B(t)
    &=\sum_{j=0}^{4}\binom{4}{j}^3 t^{3j}
      =1+64t^3+216t^6+64t^9+t^{12}.
\end{align*}
Let $\ell_k$ denote the coefficients of $L_A(t)L_B(t)$. Then $b_k(\bar{\g}_p)=\ell_k+\ell_{k-1}$, where $\bar{\g}_p$ is the associated 21-dimensional almost abelian Lie algebra. Using formula \eqref{eq: kunneth_hypercomplex}, we obtain
\begin{multline*}
b(M_p^{h})=
(1,4,22,132,389,1616,4816,9584,18197,30692,42318,54852, 62274, \\
54852,42318,30692,18197,9584,4816,1616,389,132,22,4,1).
\end{multline*}

Let us now compare these values with the full rank case for polynomials in $\Delta'_5$. For the complex case, Theorem \ref{thm: cohomology_complex} gives
\[
(1,2,1,0,0,32,64,32,0,0,1,2,1).
\]
In contrast, the additional multiplicative relations in the present example produce nonzero Betti numbers in every degree.

Similarly, in the hypercomplex case, Theorem \ref{thm: cohomology_hyp} gives Betti numbers concentrated in separated blocks of length $5$, with
\[
b_{12}=46656.
\]
In our example, all Betti numbers are nonzero and the middle Betti number increases to
\[
b_{12}=62274.
\]

Finally, the Hodge numbers of the complex solvmanifold $M_{p, \tau}^c$ with $\tau$ satisfying condition \eqref{eq: cIJ} can be found in the following diamond: 

\[
\begin{array}{ccccccccccccc}
&&&&&& 1 &&&&&& \\
&&&&& 1 && 1 &&&&& \\
&&&& 1 && 3 && 1 &&&& \\
&&& 2 && 6 && 6 && 2 &&& \\
&& 1 && 5 && 9 && 5 && 1 && \\
& 1 && 6 && 14 && 14 && 6 && 1 & \\
1 && 6 && 15 && 22 && 15 && 6 && 1 \\
& 1 && 6 && 14 && 14 && 6 && 1 & \\
&& 1 && 5 && 9 && 5 && 1 && \\
&&& 2 && 6 && 6 && 2 &&& \\
&&&& 1 && 3 && 1 &&&& \\
&&&&& 1 && 1 &&&&& \\
&&&&&& 1 &&&&&&
\end{array}
\]

\begin{remark}
    Similar examples can also be constructed in degree $6$. In particular, we have found polynomials in $\Delta'_6$ exhibiting each of the possible configurations of non-trivial multiplicative relations in the intermediate degrees. More precisely, there exist examples with:
\begin{itemize}
    \item only relations of length $3$;
    \item one relation of length $2$ together with one relation of length $4$;
    \item relations of lengths $2$, $3$, and $4$ simultaneously.
\end{itemize}
In each case, these relations occur in addition to the trivial relations of lengths $0$ and $6$.
\end{remark}

\subsection{Example with a self-reciprocal polynomial that does not satisfy the quasi full rank condition}

To illustrate the effect of arithmetic dependencies among the logarithmic parameters, we consider a self-reciprocal polynomial of degree $2$ of the form
\[
    h_m(x)=x^2-mx+1,
\]
with $m\geq 3$. For instance, let us take
\[
    h_3(x)=x^2-3x+1,
\]
and denote its roots by $r$ and $r^{-1}$.

We next construct two auxiliary polynomials whose roots are the squares and cubes of the roots of $h_3$. Recall that if
\[
    (x-r_1)\cdots(x-r_n)\in\mathbb{Z}[x],
\]
then
\[
    (x-r_1^k)\cdots(x-r_n^k)\in\mathbb{Z}[x]
\]
for every $k\in\mathbb{Z}$. A direct computation shows that the polynomial with roots $r^2$ and $r^{-2}$ is
\[
    h_7(x)=x^2-7x+1,
\]
while the polynomial with roots $r^3$ and $r^{-3}$ is
\[
    h_{18}(x)=x^2-18x+1.
\]

We define the degree $6$ polynomial
\begin{equation}\label{eq: dim6}
  p(x)=h_3(x)h_7(x)h_{18}(x).  
\end{equation}
The associated logarithmic parameters are
\[
    \log(r), \qquad \log(r^{-1}), \qquad \log(r^2), \qquad \log(r^{-2}), \qquad \log(r^3), \qquad \log(r^{-3})
\]
which are all integer multiples of $\log(r)$. Hence, vanishing sums among the logarithmic parameters reduce to integer relations inside the set
\[
    \mathcal{M}=\{\pm1,\pm2,\pm3\}.
\]
Such arithmetic dependencies generate additional zero-sum subsets, producing extra cohomology classes and therefore increasing the Betti numbers relative to the quasi full rank case.

To compute the Betti numbers of the $7$-dimensional basic solvmanifold $M_p$, we first determine the number $\ell_k$ of zero-sum subsets of $\mathcal{M}$ of cardinality $k$. Besides the trivial cancellations arising from pairs $\{m_i,-m_i\}$, the relation
\[
    1+2-3=0
\]
gives rise to two non-trivial zero-sum subsets of size three:
\[
    \{1,2,-3\},
    \qquad
    \{-1,-2,3\}.
\]
By taking complements, we obtain two additional zero-sum subsets of size four. A direct computation yields
\[
    \ell_0=1,\qquad
    \ell_1=0,\qquad
    \ell_2=3,\qquad
    \ell_3=2,\qquad
    \ell_4=3,\qquad
    \ell_5=0,\qquad
    \ell_6=1.
\]

Applying Lemma \ref{lemma: betti_combinatorial}, namely
\[
    b_k=\ell_k+\ell_{k-1},
\]
we obtain the Betti numbers of $M_p$:
\[
    b(M_p) = (1,1,3,5,5,3,1,1).
\]
This should be contrasted with the quasi full rank setting, where the absence of arithmetic dependencies would instead yield
\[
    b_3=b_4=3.
\]

Although the zero-sum subsets of $\mathcal{M}$ were counted directly in this low-dimensional example, the procedure can be systematically encoded using generating functions. This method becomes essential in higher dimensions.

For each pair $\{\pm m_i\}$, there are four possible local choices in the construction of a subset:
\[
    \emptyset,
    \qquad
    \{m_i\},
    \qquad
    \{-m_i\},
    \qquad
    \{m_i,-m_i\}.
\]
We encode the sum of the selected elements using the exponent of a variable $x$, and the cardinality of the subset using the exponent of a variable $y$. This leads to the local generating function
\[
    1+yx^{m_i}+yx^{-m_i}+y^2
    =
    (1+yx^{m_i})(1+yx^{-m_i}).
\]
Taking the product over all indices yields the global generating function
\begin{equation*}\label{eq: H_generating}
    H(x,y) = \prod_{i=1}^{3}(1+yx^i)(1+yx^{-i}) = \sum_{C\subseteq\mathcal{M}} x^{\sum_{c\in C}c} y^{|C|}.
\end{equation*}

The number $\ell_k$ of zero-sum subsets of cardinality $k$ is precisely the coefficient of $x^0y^k$ in the expansion of $H(x,y)$. Moreover, by Lemma \ref{lemma: betti_combinatorial}, the relation
\[
    b_k=\ell_k+\ell_{k-1}
\]
is encoded algebraically by multiplication with $(1+y)$, which shifts the degree in the variable $y$. Thus, the Betti numbers of $M_p$ are obtained as the coefficients of $x^0y^k$ in
\[
    U_p(x,y) = (1+y)H(x,y).
\]

Extracting these coefficients recovers exactly the sequence computed above:
\[
    (b_0,\dots,b_7) = (1,1,3,5,5,3,1,1).
\]

\subsubsection*{The complex and hypercomplex solvmanifolds}

The construction extends naturally to the complex and hypercomplex solvmanifolds by modifying the multiplicities of the logarithmic parameters and the corresponding recurrence relations.

For the complex solvmanifold $M_p^c$ of dimension $14$, with $p$ as in \eqref{eq: dim6}, each parameter appears with multiplicity two. Consequently, the local generating factors become
\[
    (1+yx^{\pm m_i})^2,
\]
while the recurrence relation \eqref{eq: kunneth_complex} contributes a global factor of $(1+y)^2$. Hence, the Betti numbers of $M_p^c$ are encoded by the generating function
\[
    U_p^c(x,y) = (1+y)^2H(x,y)^2,
\]
where $b_k^c$ is obtained as the coefficient of $x^0y^k$.

In the present example, we obtain
\begin{align*}
[x^0](U_p^c(x,y)) &= y^{14} +2y^{13} +13y^{12} +44y^{11} +111y^{10} +206y^9 +301y^8 \\
&\quad
+348y^7 +301y^6 +206y^5 +111y^4 +44y^3 +13y^2 +2y +1,
\end{align*}
and therefore
\[
    b(M_p^c) = (1,2,13,44,111,206,301,348,301,206,111,44,13,2,1).
\]

Similarly, for the hypercomplex solvmanifold $M_p^h$ of dimension $28$, with the same $p$, each parameter appears with multiplicity four. The local factors are therefore replaced by
\[
    (1+yx^{\pm m_i})^4,
\]
and the recurrence relation \eqref{eq: kunneth_hypercomplex} contributes a factor of $(1+y)^4$. Thus, the Betti numbers of $M_p^h$ are determined by
\[
    U_p^h(x,y) = (1+y)^4H(x,y)^4,
\]
with $b_k^h$ given by the coefficients of $x^0y^k$.

For the polynomial $p(x)$ considered above, this yields
\begin{align*}
[x^0](U_p^h(x,y)) &= y^{28} +4y^{27} +54y^{26} +372y^{25} +2093y^{24} +9184y^{23} +33092y^{22} +98912y^{21} \\
& \phantom{=} +249295y^{20} +536940y^{19} +995802y^{18} +1600156y^{17} +2238317y^{16} \\
& \phantom{=} +2734696y^{15} +2923220y^{14} +2734696y^{13} +2238317y^{12} +1600156y^{11} \\
& \phantom{=} +995802y^{10} +536940y^9 +249295y^8 +98912y^7 +33092y^6 +9184y^5 \\
& \phantom{=} +2093y^4 +372y^3 +54y^2 +4y +1.
\end{align*}

Hence,
\begin{align*}
b(M_p^h) = (&1,4,54,372,2093,9184,33092,98912,249295,536940,995802,1600156,\\
&2238317,2734696,2923220,\dots,1).
\end{align*}

\subsubsection*{Dolbeault cohomology of the complex solvmanifold}

When $\tau$ satisfies condition \eqref{eq: cIJ}, the Hodge numbers of the complex solvmanifold $M_{p,\tau}^{c}$ are given by the following Hodge diamond:

\[
\begin{array}{ccccccccccccccc}
&&&&&&& 1 &&&&&&& \\
&&&&&& 1 && 1 &&&&&& \\
&&&&& 3 && 7 && 3 &&&&& \\
&&&& 5 && 17 && 17 && 5 &&&& \\
&&& 5 && 27 && 47 && 27 && 5 &&& \\
&& 3 && 27 && 73 && 73 && 27 && 3 && \\
& 1 && 17 && 73 && 119 && 73 && 17 && 1 & \\
1 && 7 && 47 && 119 && 119 && 47 && 7 && 1 \\
& 1 && 17 && 73 && 119 && 73 && 17 && 1 & \\
&& 3 && 27 && 73 && 73 && 27 && 3 && \\
&&& 5 && 27 && 47 && 27 && 5 &&& \\
&&&& 5 && 17 && 17 && 5 &&&& \\
&&&&& 3 && 7 && 3 &&&&& \\
&&&&&& 1 && 1 &&&&&& \\
&&&&&&& 1 &&&&&&&
\end{array}
\]

\section{Appendix}

In this appendix we prove, using tools from number theory, that for each $n \geq 2$, the set $\Delta_n$ contains polynomials which either satisfy the full rank or the quasi full rank condition.

\begin{theorem}\label{thm:existence}
For any $n\in\N$, $n\geq 2$, there exists a full rank polynomial $p\in \Delta'_n$.
\end{theorem}

\begin{proof}
We will show that there exists a polynomial $p\in\Delta'_n$ that has exactly one root greater than 1 and all the other roots are in the interval $(0,1)$. This implies that $p$ has full rank, according to Proposition \ref{prop: tracy}$\ri$. The proof follows the lines of an argument given in \cite[Appendix]{De}.

It follows from \cite[Theorem 2.3]{DV} that there exists a totally real Galois extension $K$ of $\Q$ such that $[K:\Q]=n$. Let $\sigma_1, \sigma_2, \dots, \sigma_n \colon K \hookrightarrow \R$ denote the $n$ distinct real embeddings of $K$ into $\R$.

Let $\mathcal{O}_K^\times$ be the group of units of the ring of integers of $K$. We consider the subgroup of \textit{totally positive units}, given by
\[ \mathcal{O}_K^{\times, +} = \{ u \in \mathcal{O}_K^\times \colon \sigma_i(u) > 0 \text{ for all } i = 1, \dots, n \}. \]
By Dirichlet's Unit Theorem, the rank of $\mathcal{O}_K^\times$ is $n-1$. Since $\mathcal{O}_K^{\times, +}$ is the kernel of the signature homomorphism $s: \mathcal{O}_K^\times \to (\mathbb{Z}/2\mathbb{Z})^n$, it is a subgroup of finite index. Consequently, the rank of $\mathcal{O}_K^{\times, +}$ is also $n-1$.

Consider the Minkowski logarithmic map $L: \mathcal{O}_K^{\times, +} \to \mathbb{R}^n$ defined by:
\[ L(u) = (\log \sigma_1(u), \log \sigma_2(u), \dots, \log \sigma_n(u)) \]
The image $\Lambda = L(\mathcal{O}_K^{\times, +})$ is a lattice of rank $n-1$ embedded in the hyperplane $H \subset \mathbb{R}^n$ defined by the condition
\[ H = \{ (x_1, \dots, x_n) \in \mathbb{R}^n \colon \sum_{i=1}^n x_i = 0 \}. \]
This follows from the fact that for any $u \in \mathcal{O}_K^{\times, +}$, the norm $N(u) = \prod \sigma_i(u) = 1$, so $\sum \log \sigma_i(u) = \log(1) = 0$.

We seek a unit $u$ such that $\sigma_1(u) > 1$ and $0 < \sigma_i(u) < 1$ for $i=2, \dots, n$. In $\R^n$, this corresponds to the open cone $C$ defined by:
\[ C = \{ (x_1, \dots, x_n) \in \mathbb{R}^n \colon x_1 > 0 \text{ and } x_i < 0 \text{ for } i = 2, \dots, n \} \]
The intersection $C \cap H$ is a non-empty open cone in the $(n-1)$-dimensional space $H$. Since $\Lambda$ is a lattice of full rank in $H$, it follows from \cite[Lemma in the Appendix]{De} that $C\cap H$  must contain a point $w \in \Lambda$. By the definition of a cone, for any $k \in \mathbb{N}$, the point $k \cdot w$ also belongs to $\Lambda \cap (C \cap H)$. Thus, there exist infinitely many units $u \in \mathcal{O}_K^{\times, +}$ such that $L(u) \in C \cap H$.

To ensure that the roots are distinct, the lattice point $L(u)$ must avoid the symmetry hyperplanes defined by $x_j = x_k$ for $j \neq k$. Let $C' = (C \cap H) \setminus \bigcup_{j < k} \{x \in H \colon x_j = x_k\}$. Since we are removing a finite union of closed hyperplanes of codimension 1 from the open cone $C \cap H$, the set $C'$ remains a non-empty open subset of $H$. Assume that $\{w_1,\ldots,w_{n-1}\}$ is a basis of $H$ such that $\Lambda=\bigoplus_{i=1}^{n-1} \Z w_i$. Now, choose $v\in C'$, and write it as $v=\sum_i a_iw_i$ for some $a_i\in \R$. Since $C'$ is open, we can choose $q_i\in \Q$ sufficiently close to $a_i$, so that $v':=\sum_i q_i w_i$ is still in $C'$. Then, for some $M\in \N$ such that $Mq_i\in \Z$ for all $i$, we have that $w:=Mv'\in \Lambda \cap C'$.
By construction, the coordinates of $w$ (and consequently of all its positive integral multiples $k \cdot w$) are strictly distinct. Choosing $L(u) = w$ guarantees that the resulting unit $u$ has $n$ distinct conjugates.

Let $p$ be the minimal polynomial of this unit $u$, which has degree $n$. Then:
\begin{itemize}
    \item $p$ is monic with integer coefficients because $u$ is an algebraic integer.
    \item all roots are real and positive because $u$ is totally positive.
    \item all roots are distinct by the choice of $u$.
    \item the product of the roots is $N(u) = 1$.
    \item exactly one root $\sigma_1(u)$ is greater than 1, and the others are in $(0, 1)$ according to the definition of the cone $C$.
\end{itemize}
Thus, $p \in\Delta'_n$ and satisfies the required properties.
\end{proof}

\begin{example}
We exhibit next some examples of polynomials $p\in \Delta'_n$ satisfying the full rank condition:
\begin{align*}
p_4(x) &= x^4 -13x^3 +18x^2 -8x +1 \\
p_5(x) &= x^5 - 9x^4 + 26x^3 - 29x^2 + 11x - 1 \\
p_6(x) &= x^6 - 11x^5 + 42x^4 - 68x^3 + 46x^2 - 12x + 1 \\
p_7(x) &= x^7 - 13x^6 + 61x^5 - 131x^4 + 136x^3 - 66x^2 + 14x - 1.
\end{align*}
Indeed, $p_4$ satisfies Proposition \ref{prop: tracy}$\ri$, while $p_6$ satisfies Proposition \ref{prop: tracy}$\rii$ (it can be verified that its Galois group is $S_6$, which acts doubly transitively on its roots). On the other hand, the fact that $p_5$ and $p_7$ have full rank follows from Proposition \ref{prop: tracy}$\riii$.  
\end{example}

\medskip

In order to verify the existence of polynomials in $\Delta_n'$ satisfying the quasi full rank condition, we recall the following definition.

\begin{definition}\label{def: linearly_disjoint}
The quadratic extensions $\Q(\sqrt{q_1}),\ldots, \Q(\sqrt{q_n})$ with $q_i\in \N$ are said to be \textit{linearly disjoint} if
\[
[\Q(\sqrt{q_1}, \ldots, \sqrt{q_n}) \colon \Q]
= \prod_{i=1}^n [\Q(\sqrt{q_i}) \colon \Q]=2^n.
\]
Equivalently, for every nonempty subset $I \subset \{1,\dots,n\}$, the product $\prod_{i \in I} q_i$ is not a square in $\N$.
\end{definition}

\begin{theorem} \label{thm:disjoint_extensions}
    Let $\{h_{k_i}(x) = x^2 - k_i x + 1\}_{i=1}^m$ be a set of self-reciprocal polynomials with $k_i \in \N, k_i \ge 3$. Let $r_i > 1$ be the largest root of $h_{k_i}$ for each $i$, and let $s_i = \log r_i$ be the corresponding logarithmic parameters. If the quadratic extensions of $\Q$, $K_i = \Q(\sqrt{k_i^2 - 4})$, are linearly disjoint over $\Q$, then the polynomial $p = \prod_{i=1}^m h_{k_i}$ lies in $\Delta'_{2m}$ and has quasi full rank.
\end{theorem}

\begin{proof}
It is clear that $p \in \Delta'_{2m}$ is self-reciprocal. We now show that $p$ satisfies the quasi full rank condition, i.e., that any multiplicative relation as in \eqref{eq: qfr_multiplicative_r}
\begin{equation} \label{eq:mult_relation}
    r_1^{a_1} r_2^{a_2} \cdots r_m^{a_m} = 1
\end{equation}
forces $a_i = 0$ for all $i$. 

Consider the field 
\[
L = \Q(\sqrt{k_1^2 - 4}, \sqrt{k_2^2 - 4}, \dots, \sqrt{k_m^2 - 4})
\]
Since the quadratic fields $K_i$ are linearly disjoint, the Galois group $\text{Gal}(L/\Q)$ is isomorphic to $(\Z/2\Z)^m$. Specifically, for each $j \in \{1, \dots, m\}$, there exists an automorphism $\sigma_j \in \text{Gal}(L/\Q)$ such that $\sigma_j$ acts as the non-trivial automorphism on $K_j$ (mapping $\sqrt{k_j^2-4} \mapsto -\sqrt{k_j^2-4}$) and fixes $K_i$ for all $i \neq j$.

Recall that the root $r_j$ is given by $r_j = \frac{1}{2}(k_j + \sqrt{k_j^2-4})$, so $\sigma_j(r_j) = \frac{1}{2}(k_j - \sqrt{k_j^2-4}) = r_j^{-1}$. Applying $\sigma_j$ to \eqref{eq:mult_relation} gives
\[
    r_j^{-a_j} \prod_{i \neq j} r_i^{a_i} = 1.
\]
Dividing the original relation \eqref{eq:mult_relation} by this, all terms $r_i$ for $i \neq j$ cancel out, leaving $r_j^{2a_j} = 1$. Since $r_j > 1$, it must be $a_j = 0$. As this argument holds for every $j$, we conclude that $a_i = 0$ for all $i$, and therefore $p$ has quasi full rank.
\end{proof}

\begin{corollary} \label{lemma: coprime}
With notation as in Theorem \ref{thm:disjoint_extensions}, let $d_i$ denote the square-free part of $k_i^2-4$. If the numbers $d_i$ are pairwise coprime then the quadratic extensions $K_i=\Q(\sqrt{k_i^2-4})=\Q(\sqrt{d_i})$ are linearly disjoint over $\Q$, and hence Theorem \ref{thm:disjoint_extensions} holds.
\end{corollary}

\begin{proof}
    It is clear that no nontrivial product of the $d_i$'s is a square in $\N$, hence the extensions are linearly disjoint.
\end{proof}

\end{document}